\documentclass[preprint,3p]{elsarticle}

\usepackage{amssymb,amsmath,amsthm}
\newtheorem{theorem}{Theorem}
\newtheorem{lemma}{Lemma}
\newtheorem{corollary}{Corollary}
\newtheorem{definition}{Definition}

\usepackage{mathtools} % Bonus

\usepackage{xcolor}
\usepackage{bm}
\usepackage{xcolor}
\usepackage{amssymb}
\usepackage{amsfonts}
\usepackage{amsmath}
\usepackage{fancybox}
\usepackage{color}
\usepackage{bm}
\newtheorem{remark}{Remark}

\usepackage{graphicx}
\usepackage{epstopdf}
\newcommand{\DIVGMLS}{\overset{{\mbox{\tiny\sf GMLS}}}{DIV}} 
\newcommand{\DIVGMLSP}{\overset{{\mbox{\tiny\sf GMLS$^\star$}}}{DIV}} 

\DeclareMathAccent{\maxvec}{\mathord}{letters}{"7E}
\DeclareMathOperator*{\argmin}{argmin}
\journal{Journal}

\begin{document}

\begin{frontmatter}


                    
\title{A conservative, consistent, and scalable meshfree mimetic method}


\cortext[cor1]{Corresponding author}

\fntext[sand-blurb]{Sandia National Laboratories is a multimission laboratory managed and operated by National Technology and Engineering Solutions of Sandia, LLC., a wholly owned subsidiary of Honeywell International, Inc., for the U.S. Department of Energy's National Nuclear Security Administration under contract DE-NA-0003525. This paper describes objective technical results and analysis. Any subjective views or opinions that might be expressed in the paper do not necessarily represent the views of the U.S. Department of Energy or the United States Government.
}

\fntext[CCR]{Center for Computing Research, Sandia National Laboratories, Albuquerque, NM 87185}

\author[CCR,sand-blurb]{Nathaniel Trask\corref{cor1}}
\ead{natrask@sandia.gov}
\ead[url]{http://www.sandia.gov/~natrask/}

\author[CCR,sand-blurb]{Pavel Bochev}
\ead{pbboche@sandia.gov}
\ead[url]{http://www.sandia.gov/~pbboche/}

\author[CCR,sand-blurb]{Mauro Perego}
\ead{mperego@sandia.gov}
\ead[url]{http://www.sandia.gov/~mperego/}

\begin{abstract}
Mimetic methods discretize divergence by restricting the Gauss theorem to mesh cells. Because point clouds lack such geometric entities, construction of a compatible meshfree divergence remains a challenge.
In this work, we define an abstract Meshfree Mimetic Divergence (MMD) operator on point clouds by contraction of \emph{field} and \emph{virtual face} moments. This MMD satisfies a discrete divergence theorem, provides a discrete local conservation principle, and is first-order accurate. 
We consider two MMD instantiations. The first one assumes a background mesh and uses 
generalized moving least squares (GMLS) to obtain the necessary field and face moments. This MMD instance is appropriate for settings where a mesh is available but its quality is insufficient for a robust and accurate mesh-based discretization. 
The second MMD operator retains the GMLS field moments but defines \emph{virtual face} moments using computationally efficient weighted graph-Laplacian equations. This MMD instance does not require a background grid and is appropriate for applications where mesh generation creates a computational bottleneck. It allows one to trade an expensive mesh generation problem for a scalable algebraic one, without sacrificing compatibility with the divergence operator.
We demonstrate the approach by using the MMD operator to obtain a virtual finite-volume discretization of conservation laws on point clouds. 
Numerical results in the paper confirm the mimetic properties of the method and show that it behaves similarly to standard finite volume methods.
\end{abstract}

\begin{keyword}
Meshfree \sep
Generalized moving least squares \sep
Compatible discretizations\sep
Mimetic methods



\end{keyword}

\end{frontmatter}



\section{Introduction}
The vast majority of numerical methods for  PDEs rely on partitioning of the spatial domain into a mesh to both represent the unknown fields and to define the discrete operators acting on these fields. The quality of the resulting mesh-based schemes usually depends on the quality of the underlying mesh and may suffer when the latter deteriorates. For example, shape functions defined on ``sliver'' elements can result in badly conditioned or even singular stiffness matrices; see  \cite{Babuska_76_SINUM} and the more recent survey \cite{Shewchuk_02_INPROC}.

Accordingly, automated generation of high-quality grids is a task of tremendous practical importance for mesh-based PDE discretizations. Yet, for problems with complex geometries this task remains a major challenge, and creates a significant performance bottleneck in the simulation workflow \cite{Hardwick_05_TECHREPORT}.  As future applications continue to evolve beyond forward simulation, overcoming this bottleneck will be of essence for enabling tasks such as shape optimization and uncertainty quantification.
For applications where the domain evolves in time, maintaining a high-quality mesh is even more challenging, typically requiring the introduction of remap/remeshing algorithms as elements collapse and tangle under Lagrangian motion. This situation is typical of continuum mechanics problems characterized by large deformation and topology change, such as free-surface flows, hypervelocity impact, and ductile fracture. 
%In areas such as uncertainty quantification, where quadrature in dimension $d>>3$ is required, the computational complexity of mesh generation grows exponentially with $d$.

Meshfree methods can significantly alleviate or even completely eliminate the mesh generation bottleneck. 
For example, the domain integration in Galerkin meshfree methods can be performed on substandard meshes, by overlaying the domain with a mesh that is independent of the point cloud \cite{Belytschko_94_IJNME}, or even by a non-conforming set of cells as in the SNNI family of schemes  \cite{Chen_13_IJNME}. 
Collocation meshfree methods on the other hand only require a formation of well-distributed point clouds, which is much easier than generation of high-quality grids, especially for complex geometric domains.

%
%Particle and meshfree methods offer a framework where a problem may be described as an evolving set of 0-simplices (or points) with no reference to higher order $k-simplices (such as edges, faces and cells), and thus have no notions of shape-regularity. 
%\PB{However, the absence of an underlying mesh deprives meshfree methods of the topological structures and the metric information that is the basis for compatible mesh-based discretizations.}  
%
However, since their inception, meshfree methods have struggled to simultaneously maintain rigorous notions of consistency, accuracy, stability, and compatibility that are now commonplace in mesh-based approaches. While significant improvements have been made in recent years addressing the approximation theory of meshfree methods, a meshfree framework mirroring the conservation  and consistency properties of mesh-based compatible discretizations has remained elusive.
The root cause is that the topological and geometric data employed by typical compatible mesh-based discretizations such as mimetic finite differences (MFD) \cite{Lipnikov_13_JCP}, Whitney elements \cite{Bossavit_88_IEE}, and finite volumes  \cite{Eymard_00_INCOLL,Trapp_06_THESIS,Nicolaides_97_SINUM} cannot be assumed readily available in the meshfree setting, especially in the absence of a conforming background mesh. 
To successfully transplant the properties of these methods to such a setting it is important to recognize that their compatibility is a consequence of discrete differentiation based on the generalized Stokes theorem
\begin{equation}\label{eqn:GST}
  \int_\omega d\bm{u} = \int_{\partial \omega} \bm{u} \,.
\end{equation}
In \eqref{eqn:GST} $d$ is the exterior derivative,  $\bm{u}$ is a $k$-differential form and $\omega$ is a $k+1$ dimensional manifold. This theorem implies that for a sufficiently smooth $\bm{u}$ its derivative at a point $\bm{x}$ can be characterized as
\begin{equation}\label{eq:d}
d \bm{u} (\bm{x}) = \lim_{ \mu(\omega)\rightarrow 0}
\frac{\int_{\partial \omega} \bm{u} } { \mu(\omega)} \,,
\end{equation}
where $\omega$ is a $k+1$-dimensional region with measure $\mu(\omega)$ containing the point $\bm{x}$. Restriction of the right hand side in \eqref{eq:d} to a $k+1$-dimensional mesh entity $\omega_h$ then yields a notion of a discrete derivative  that satisfies a discrete Stokes theorem by construction. In particular, the resulting discrete grad, curl and div operators mimic the vector calculus identities $curl\, grad\, \bm{u} = 0$ and $ div\, curl \,\bm{u}=0$; see, e.g., \cite{Hyman_97_CMA}.

%
%For the successful transplantation of these ideas into a meshfree setting it is important to recognize that the centerpiece of the compatible discretizations described above is a discrete derivative definition based on the generalized Stokes theorem. Specifically, \eqref{eqn:GST} implies that for a sufficiently smooth $\bm{w}$ its derivative at a point $\bm{x}$ can be characterized as
%%
%\begin{equation}\label{eq:d}
%d \bm{w} (\bm{x}) = \lim_{ \mu(\omega)\rightarrow 0}
%\frac{\int_{\partial \omega} \bm{w} } { \mu(\omega)}
%\end{equation}
%%
%where $\omega$ is a bounded region with measure $\mu(\omega)$ containing the point $\bm{x}$. Restriction of the right hand side in \eqref{eq:d} to the mesh cells then yields the notion of a discrete derivative \eqref{eq:stokes-dis} that satisfies a discrete Stokes theorem by construction.

If one wishes to use this procedure to define the discrete derivatives, then the underlying discretization infrastructure must provide the following three key pieces of information:
\begin{itemize}
\setlength\itemsep{0.5ex}
\item \emph{field data} representing the integrand $\bm{u}$;
\item \emph{topological data} describing the action of the boundary operator on a geometric entity $\omega$, and
\item \emph{metric data} providing the measures of the geometric entity $\omega$ and its boundary $\partial\omega$. 
\end{itemize}
The first one, i.e., the field data, can be presumed available irregardless of the type of the underlying discretization as both mesh-based and meshfree methods should at a minimum be able to represent the approximate PDE solutions. 
%In particular, in this paper we use the generalized moving least-squares framework (GMLS) \cite{WendlandBook} to generate the necessary field data.
%
The second and third information pieces are virtually always available in mesh-based methods. The nodes, edges, faces and cells in most grids form a chain complex, which provides the necessary topological data for compatible discrete grad, curl and div operators; see, e.g., \cite{Bochev_06a_IMA,matt97}. Likewise, the necessary metric data can be calculated from the mesh description without much difficulty. 

This is not the case when the spatial domain is represented by a point cloud. The amount of metric and topological information that can be extracted from the cloud at a reasonable cost is quite limited. One can build, at a linear cost, the $\epsilon$-ball graph of the cloud, construct its vertex-to-edge incidence matrix and compute the lengths of its edges. This would provide \eqref{eq:d} with just enough information to build a compatible discrete gradient \cite{Bochev_06a_IMA} on the point cloud but not much else.

Here we develop a computationally efficient and scalable approach to generate the topological and metric data for \eqref{eq:d} needed  to define a compatible divergence operator on point clouds. 
Using as a template a mimetic divergence on a primal-dual mesh, we employ a \emph{virtual} primal-dual mesh to construct 
an abstract meshfree version of this operator by contraction of \emph{field moments}, associated with the virtual faces, and \emph{face moments}, which characterize the latter algebraically. 
%
%One possible way to accomplish this would be to extend  the point cloud to include proper cells and faces.  However, this violates the main premise of meshfree methods that the problem be entirely characterized by points and values associated with these points to avoid the complexity of mesh generation and maintenance. And since building cells and faces is effectively a meshing process, it would reinstate the meshing bottleneck that we hoped to avoid to begin with.
%To solve this dilemma we use a mimetic divergence operator constructed on a primal-dual mesh structure as a template for an abstract \emph{meshfree mimetic divergence operator} (MMD). 
%
%A data transfer operator maps the native field values stored on the point cloud into these field moments. We derive conditions on the face and field moments and the data transfer that ensure local conservation and first-order accuracy of the abstract operator.

We consider two instantiations of this abstract mimetic meshfree divergence (MMD) operator. The first one assumes a background mesh and uses  generalized moving least squares (GMLS) \cite{WendlandBook} to obtain the necessary field and face moments. This MMD instance is appropriate for settings where a background mesh exists but its quality is insufficient for a robust and accurate mesh-based discretization. 
In this case the mesh is used only to define a boundary operator and to integrate the local GMLS polynomial basis over the cell faces. Both of these tasks can be performed reliably even on substandard grids.   

The second MMD operator does not assume a background mesh and uses instead the $\epsilon$-ball graph of the point cloud and its formal dual as surrogates for a primal and a virtual dual mesh, respectively. The  face-to-cell incidence matrix of the latter  provides the necessary topological information for the divergence operator.  This MMD instance retains the GMLS field moments but defines the virtual face moments in terms of \emph{virtual area potentials}, leading to a graph Laplacian problem that  can be solved in an efficient and scalable manner by standard multigrid preconditioners.   In so doing we trade a challenging mesh generation problem for a scalable algebraic one, without sacrificing compatibility with the divergence operator.

In \cite{trask_17_sisc} we used the \emph{local} connectivity graph of each particle and its formal dual to mimic the staggered arrangement in div-grad stencils on primal-dual grids. The resulting ``staggered'' meshfree scheme behaves similarly  to mesh-based div-compatible discretizations but is not locally conservative because the primal-dual mesh surrogate is defined independently for each particle. This paper is a logical extension of \cite{trask_17_sisc} aiming to deliver a  meshfree  scheme mirroring the conservation properties of traditional compatible discretizations.

While there is an abundant literature on compatible mesh-based discretizations, this is not the case for meshfree methods. To the best of our knowledge, the extant work comprises the meshless volume scheme \cite{Katz_09_AIAA}, the Uncertain Grid Method (UGM) \cite{Diyankov_08_TECHREP}, the meshfree framework for conservation laws in \cite{Jameson_12_SISC}, and a few related references. 
Of particular interest to us are \cite{Diyankov_08_TECHREP} and \cite{Jameson_12_SISC} because, as in the second MMD instance, they define the necessary metric data in a purely algebraic fashion without assuming a background grid. 

The UGM appears to be the first example of a locally conservative mesfree scheme, which uses virtual geometric entities characterized  algebraically by numbers representing their ``measures''.  In UGM, the term ``uncertain'' refers to an ``uncertain face'' between two neighboring particles, which is similar to our virtual faces.
The meshfree framework in \cite{Jameson_12_SISC} also seeks sets of numbers that can be interpreted as measures of virtual volumes, centered at each particle, and the areas of their virtual faces, respectively. 
The single most important difference between these works and our approach is in the type of algebraic problems that generate the necessary metric data. 
Both \cite{Diyankov_08_TECHREP} and \cite{Jameson_12_SISC}  find this data  by solving global constrained optimization problems. In UGM this problem is a linear program solved by a primal-dual logarithmic barrier method, whereas in \cite{Jameson_12_SISC}  it is a quadratic program (QP) requiring a QP-specific solver.
In contrast, our approach requires solution of several graph Laplacian problems, which can be accomplished in linear time.  

We have organized the paper as follows. Section \ref{sec:notation} introduces notation and reviews the necessary technical background. 
Section \ref{sec:abstract} formulates the abstract MMD operator and states general conditions for its consistency and accuracy, while Section \ref{sec:MMD-instances} presents its two instances. Efficient  and scalable computation of the virtual face moments necessary for the second MMD instance is discussed in Section \ref{sec:metric}.
In Section \ref{sec:apply} we use the MMD operator to discretize a model conservation law problem and then in Section \ref{sec:num} we provide some representative numerical examples.
We summarize our findings and discuss future work in Section \ref{sec:concl}.

%%%%%%%%%%%%%%%%%%%%%%%
\section{Preliminaries}\label{sec:notation}
%%%%%%%%%%%%%%%%%%%%%%%
%%%%%%%%%%
\subsection{Notation}\label{sec:notation-general}
%%%%%%%%%%
Throughout the paper upper case fonts are reserved for function spaces, operators and sets of various entities, while lower case fonts stand for scalar fields, linear functionals, indices, etc.  
We denote the standard Euclidean norm on $\mathbb{R}^n$ by $|\cdot|$ and use bold face fonts to denote vector quantities, e.g., $\bm{x}=(x_1,\ldots,x_d)$ is a point in the Euclidean space $\mathbb{R}^d$, $\bm{e}_{ij}=(\bm{x}_i,\bm{x}_j)$ is an edge connecting two such points, $\bm{u}=(u_1,\ldots,u_d)$ is a vector field in $\mathbb{R}^d$, and $\bm{n}=(n_1,\ldots,n_d)$ is a unit normal vector, i.e., $|\bm{n}| = 1$. 

As usual,  $C^k(\Omega)$ is the space of all $k$-continuously differentiable functions, and $P_m(\mathbb{R}^d)$, or simply $P_m$ is the space of all multivariate polynomials of degree less than or equal to $m$. 
We denote the standard norm on $C^k(\Omega)$ and its restriction to $\omega\subset\Omega$ by
$
\| \cdot \|_{C^k(\Omega)}
$
and
$
\| \cdot \|_{C^k(\omega)},  
$
respectively.

%%%%%%%%%%%%%%%%
\begin{figure}[h!]
  \centering
  \includegraphics[height=1.5in]{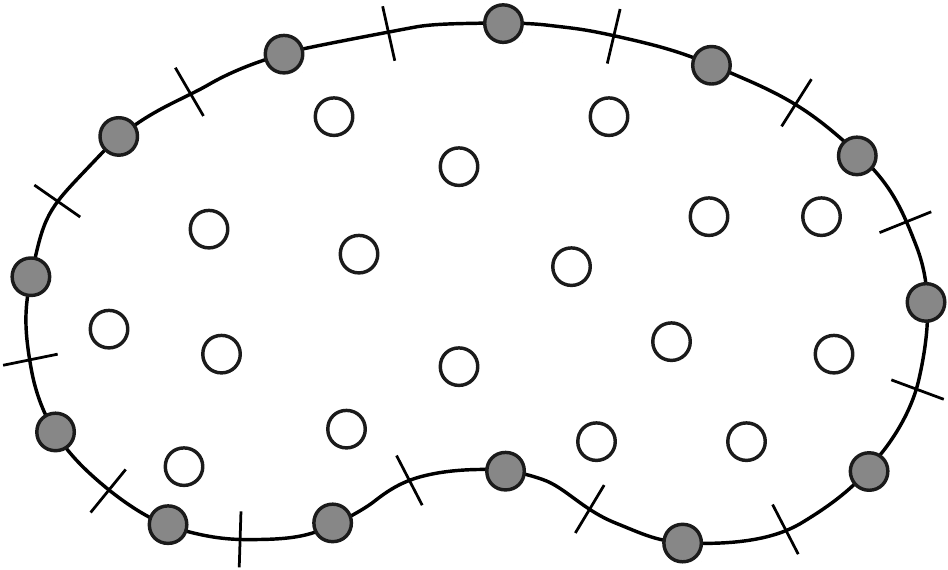} 
  \caption{Point cloud discretization of a domain $\Omega$. The solid dots represent the boundary points forming the set $X_\Gamma$ whereas the circles represent the interior points forming the set $X_0$.
}
  \label{fig:domain}
\end{figure}
%%%%%%%%%%%%%%%%%

Let $\Omega \subset \mathbb{R}^d$, $d=2,3$ be a bounded connected region with a Lipschitz continuous boundary $\Gamma=\partial \Omega$. We assume that $\Gamma$ may be partitioned into a collection of $p_{\Gamma}$ faces $\Gamma = \cup \Gamma_j$, each having maximal diameter $h$, measure $\mu_{\Gamma_j}$, centroid $\bm{x}_j$ and outward unit normal $\bm{n}_j$. The collection of these centroids forms the set ${X}_{\Gamma}$ of all \textit{boundary particles}. In the interior of the domain, we assume a set ${X}_0$ containing $p_0$ \textit{interior particles}. The union ${X} = {X}_0 \cup {X}_{\Gamma}$ has $p_{\Omega}=p_0+p_{\Gamma}$ particles and defines the meshfree discretization of $\Omega$. 
Since by assumption $\Omega$ is bounded we can choose the coordinate system in such a way that the point cloud is strictly contained in the first quadrant, i.e., ${x}_k > 0$, $k=1,\ldots,d$ for all $\bm{x} \in {X}$. We denote the subset of $X$ contained in an entity $\omega$ by $X_\omega := \omega\cap X$. Whenever required, boundary conditions will be imposed on  simply connected non-empty subsets $\gamma\subseteq\Gamma$.  Without loss of generality we shall assume that  $\gamma$ and its complement $\gamma^\prime:=\Gamma\setminus\gamma$ are exact unions of boundary segments $\Gamma_j$ and define their indicator functions as
$$
\chi_i (\gamma) = 
\left\{
\begin{array}{rl}
\Gamma_i & \mbox{if $\bm{v}_i\in\gamma$} \\[0.5ex]
\emptyset  & \mbox{otherwise}
\end{array}
\right. 
\quad\mbox{and}\quad
\chi_i (\gamma^\prime) = 
\left\{
\begin{array}{rl}
\Gamma_i & \mbox{if $\bm{v}_i\in\gamma^\prime$} \\[0.5ex]
\emptyset  & \mbox{otherwise}
\end{array}
\right. ,
$$
respectively. The quality of the point cloud ${X}$ can be characterized by its \textit{fill distance} \cite[p.14]{WendlandBook} given by 
\begin{equation}\label{eq:fill}
h_{X} = 
\sup_{\bm{x} \in \Omega}
\min_{\bm{x}_i \in {X}}
%\underset{\bm{x} \in \Omega}{\sup}\, \underset{x_i \in {X}}{\min}
| \bm{x} - \bm{x}_i |,
\end{equation}
and its \textit{separation distance} \cite[p.41]{WendlandBook} defined as 
$$
q_{{X}} = \frac12 \min_{i \neq j} | \bm{x}_i - \bm{x}_j | ,
$$
We shall assume that ${X}$ is \textit{quasi-uniform}, namely that there exists $c_{qu} > 0$ such that
\begin{equation}\label{eq:quasi}
q_{{X}} \leq h_{X} \leq c_{qu} q_{{X}}.
\end{equation}

We denote the spaces of all discrete scalar and vector fields on $X$ by $S^h$ and $U^h$, respectively. These fields represent functions by their point samples on $X$, which we group as
$$
\begin{aligned}
S^h\ni u^h & =\{ u_{i} \,|\, \bm{x}_i \in X_0\}\cup 
\{ u_{i} \,|\, \bm{x}_i \in X_{\gamma^\prime}\} 
\cup \{ u_{i} \,|\, \bm{x}_i \in X_{\gamma}\}
\in \mathbb{R}^{p_\Omega }, \quad\mbox{and} 
\\[0.5ex]
U^h \ni \bm{u}^h & =\{ \bm{u}_{i} \,|\, \bm{x}_i \in X_0\}\cup 
\{ \bm{u}_{i} \,|\, \bm{x}_i \in X_{\gamma^\prime}\} \cup
\{ \bm{u}_{i} \,|\, \bm{x}_i \in X_{\gamma}\}
\in \mathbb{R}^{dp_\Omega} \,,
\end{aligned}
$$
respectively. Note that $\bm{u}^h=(u^h_1,\ldots,u^h_d)$, where $u^h_k\in S^h$, $k=1,2,\ldots, d$. 
We denote by $S^h_{\gamma}$ and $U^h_\gamma$ the sets of all discrete scalar and vector fields defined by their point values in $X_0\cup X_{\gamma^\prime}$, i.e.,
$$
\begin{aligned}
S^h_\gamma\ni u^h & =\{ u_{i} \,|\, \bm{x}_i \in X_0\}\cup 
\{ u_{i} \,|\, \bm{x}_i \in X_{\gamma^\prime}\} 
\in \mathbb{R}^{p_\Omega - p_{\gamma}}, \quad\mbox{and} 
\\[0.5ex]
U^h_\gamma\ni \bm{u}^h & =\{ \bm{u}_{i} \,|\, \bm{x}_i \in X_0\}\cup 
\{ \bm{u}_{i} \,|\, \bm{x}_i \in X_{\gamma^\prime}\} 
\in \mathbb{R}^{d(p_\Omega - p_{\gamma})} \,.
\end{aligned}
$$
We use $S^h_{\gamma}$ or $U^h_\gamma$  whenever a scalar field $u(\bm{x})$ or a vector field  
$\bm{u}(\bm{x})$ is prescribed on $\gamma$. Let 
$$
\| u^h \|_{\ell^\infty, X_\omega} := \max_{\bm{x}_i\in X_\omega} |u^h_i | \,.
$$
We equip $S^h$ and $S^h_\gamma$ with the norms
$$
\| u^h \|_{S^h}  := \|u^h \|_{\ell^\infty, X}
\quad\mbox{and}\quad
\| u^h \|_{S^h_\gamma} := \|u^h \|_{\ell^\infty, X\setminus X_{\gamma}}
$$
respectively, and similarly for $U^h$ and $U^h_\gamma$. 
%

%%%%%%%%%%%%%%%%%%%%%%%
\subsection{GMLS essentials} \label{sec:GMLS}
%%%%%%%%%%%%%%%%%%%%%%%

Generalized moving least squares (GMLS) is a non-parametric regression technique for the approximation of bounded linear functionals from scattered samples of their arguments \cite{WendlandBook}. GMLS has been used in meshfree collocation schemes  to approximate point values of derivatives; see  \cite{mirzaei2011}, \cite{Mirzaei_12_IMAJNA}, \cite{trask_17_sisc} and the references therein. 
Here we shall apply GMLS to approximate integrals as our goal is to develop a meshfree  divergence operator based on the definition \eqref{eq:d}.
We provide a brief summary of the GMLS framework, specialized to our needs. The abstract GMLS setting comprises the following key ingredients; see \cite[Section 4.3]{WendlandBook}:
\begin{itemize}
\setlength\itemsep{0.4ex}
\item a function space $U$ with a dual $U^*$;
\item a finite dimensional space $\Phi=\mbox{span}\{\phi_1,\ldots,\phi_q\}\subset U$; 
\item a finite set of linear functionals $\Lambda=\{\lambda_1,\ldots,\lambda_n\}\subset U^*$; and 
\item a correlation (kernel) function $w:U^*\times U^*\mapsto \mathbb{R}^+\cup\{0\}$.
\end{itemize}
The set $\Lambda$ is assumed to be $\Phi$-unisolvent, that is
\begin{equation}\label{eq:P-uni}
\{\phi\in \Phi\,|\, \lambda_i(\phi) = 0, i=1,\ldots,n\} = \{0\}.
\end{equation}
Given a target functional $\tau\in U^*$ and an arbitrary $u\in U$, GMLS seeks an approximation $\widetilde{\tau}(u)$ of $\tau(u)$ in terms of the sample set $\bm{\lambda}(u)=(\lambda_1(u),\ldots,\lambda_n(u))\in \mathbb{R}^n$, such that  
$\widetilde{\tau}(\phi) = \tau(\phi)$ for all $\phi\in\Phi$. We call such a GMLS approximation $\Phi$-reproducing.
To describe the GMLS solution to this problem we introduce the vector $\bm{\tau}(\phi)\in \mathbb{R}^q$ with elements
$$
(\bm{\tau}(\phi))_i = \tau(\phi_i),\quad i=1,\ldots, q,
$$
the diagonal weight matrix  $W(\tau)\in \mathbb{R}^{n\times n}$ with element $W_{ii}(\tau)=w(\tau;\lambda_i)$, and the basis sample matrix $B \in \mathbb{R}^{n\times q}$ with element
$$
B_{ij} = \lambda_{i}(\phi_j); \quad i=1,\ldots, n; \ \ j=1,\ldots,q .
$$
In what follows, $|\cdot |_{W(\tau)}$ denotes the Euclidean norm on $\mathbb{R}^n$ weighted by $W(\tau)$, i.e.,
$$
|\bm{b} |^2_{W(\tau)} = \bm{b}^{\intercal} W(\tau) \bm{b} \quad \forall \bm{b}\in\mathbb{R}^n \,.
$$
One can then show that the $\Phi$-reproducing GMLS approximation of the target functional is given by
\begin{equation}\label{taudef}
\widetilde{\tau}(u) := \bm{c}(u)\cdot \bm{\tau}(\phi),
\end{equation}
where the GMLS coefficient vector $\bm{c}(u) \in\mathbb{R}^q$ solves the weighted least-squares problem
\begin{equation}\label{eq:gmls-b}
\bm{c}(u)
= \argmin_{\bm{c}\in\mathbb{R}^q}\frac12
\left | B\bm{c} -  \bm{\lambda}(u)  \right |^2_{W(\tau)}
\end{equation}
It is not hard to see that
\begin{equation}\label{gmlscoeffs}
\bm{c}(u) =\left(B^{\intercal}W(\tau)B\right)^{-1}  (B^\intercal W(\tau))\bm{\lambda}(u)\,.
\end{equation}
 We refer to previous work \cite{trask_17_sisc} for details regarding the efficient and stable calculation of $\bm{c}(u)$.

%In this paper we use GMLS to approximate integrals over compact regions $\omega$ with centroids $\bm{x}_{\omega}$. Thus, our target functionals are given by 
%%
%\begin{equation}\label{eq:target}
%\tau_{\omega}(u) = \int_{\omega} u\, d\mu \,.
%\end{equation} 
%%
%For the approximation of these targets we consider the following GMLS specialization:
%%
%\begin{itemize}
%\item the function space $U$ is $C^{1}(\Omega)$ or its vector-valued analogue $(C^1(\Omega))^d$; 
%\item the finite dimensional space $\Phi=P_1$; 
%\item the sampling functionals are point evaluations of vector or scalar fields, i.e., $\lambda_i = \delta_{\bm{y}_i}$ where $\bm{y}_i\in\mathbb{R}^d$ are given sampling points;
%%
%\item the weight function $w(\tau_{\omega};\lambda_i) = \kappa(|\bm{x}_{\omega} - \bm{y}_i |)$, where $\kappa(r)$ is a positive, radially symmetric kernel with support in $(-\epsilon_\kappa,\epsilon_\kappa)$.  
%\end{itemize}
%%
%In this work we use the kernel function  $\kappa(r) = \left(1-{r}/{\epsilon_\kappa}\right)^p_+$ with $p = 4$. 
%We assume $\epsilon_\kappa$ is large enough to obtain $P_1$-unisolvency. 
%%
%With such a choice  the GMLS approximation of the target functional is guaranteed to be exact for linear polynomials. We refer to \cite[Chapter 3]{WendlandBook} for conditions on the sampling locations $\{\bm{y}_i\}$ which are sufficient to ensure polynomial reproduction and rigorous approximation error estimates. 

%%%%%%%%%%%%%%%%%%%%%%%
\subsection{Mimetic divergence operator}\label{sec:mimeticdiv}
%%%%%%%%%%%%%%%%%%%%%%
In Section \ref{sec:abstract} we define an abstract meshfree  divergence operator  by mimicking a mimetic divergence operator on a primal-dual mesh \cite{Nicolaides_92_SINUM,Nicolaides_97_SINUM,Desburn_05_ARXIV}.
Remark \ref{rem:primal_dual} explains the appropriateness of this choice. To describe the template operator assume  that  $\Omega$ is equipped with discretization infrastructure comprising
\begin{itemize}
\setlength\itemsep{0.3ex}
\item a primal mesh $\tau_h(V,E)$ with vertices $V=\{\bm{v}_i\}_{i=1}^{p_V}$ and edges $E=\{\bm{e}_{ij}\}$; $\bm{e}_{ij} = (\bm{v}_i,\bm{v}_j) \in V\times V$,
\item a topologically dual mesh $\tau_h^{\prime}(C,F)$ with cells $C=\{\omega_i\}_{i=1}^{p_V}$ and affine faces $F=\{\bm{f}_{ij}\}$, and
\item a pair of discrete spaces $\widetilde{U}^h_\gamma$ and $\widetilde{S}^h$ with elements
\begin{equation}\label{eq:dual_fields}
\bm{u}^h =\{ u_{ij}\in\mathbb{R}\,|\, \bm{f}_{ij}\in F\} \cup \{u_{\Gamma_i}\in\mathbb{R}\,|\, \Gamma_i\in \gamma^\prime\}
\quad\mbox{and}\quad
{s}^h =\{ \phi_i \in \mathbb{R}\,|\, \omega_i\in C\},
\end{equation}
respectively, representing the domain and the range of the discrete divergence, respectively. 
\end{itemize}
Every dual cell $\omega_i$ corresponds to a primal vertex $\bm{v}_{i}$ and has a measure $\mu_i$. We assume that every dual face $\bm{f}_{ij}$ corresponds to a primal edge $\bm{e}_{ij}$, has a unit normal $\bm{n}_{ij} = \bm{e}_{ij}/|\bm{e}_{ij}|$, and oriented measure $\mu_{ij} = \int_{\bm{f}_{ij}}dS$.
%and its barycenter $\bm{v}_{ij}$ coincides with the midpoint of $\bm{e}_{ij}$. 
If $\omega_i$ and $\omega_j$ are the dual cells corresponding to the endpoints of this edge, then $\bm{f}_{ij} = \partial\omega_i\cap\partial\omega_j$. 
%%
%\begin{remark}\label{rem:primal-dual}
%One example of  a primal-dual grid structure satisfying these assumptions is provided by a dual mesh given by a centroidal Voronoi tessellation \cite{Du_99_SIREV} of $\Omega$ and  a primal mesh defined as the associated Delaunay triangulation.
%\end{remark}
%
The space $\widetilde{U}^h_\gamma$ approximates vector fields in $H^0_\gamma(div,\Omega)=\{\bm{u}\in H(div,\Omega)\cap (C^0(\Omega))^d\,|\,  \mbox{$\bm{u}=\bm{u}_\gamma$ on $\gamma$}\}$ by their scalar projections onto $\bm{n}_{ij}$ or $\bm{n}_{ji}$, whereas $\widetilde{S}^h$ approximates $L^2(\Omega)$ functions  by their dual cell averages. Without a loss of generality we shall assume that the restriction $\mathcal{R}(\bm{u})\in \widetilde{U}^h_\gamma$ of a vector field $\bm{u}$ is given by $u_{ij}=u_{ji} = \bm{u}\cdot\bm{n}_{ij}$ for $i<j$.
The neighborhood of $\bm{v}_i$ is defined as 
$$
V_{i} = \{ \bm{v}_j\in V\,|\, \bm{e}_{ij}=(\bm{v}_i,\bm{v}_j)\in E\} 
$$ 
and contains all mesh vertices connected to $\bm{v}_i$ by an edge. 
%
%
%A non-empty subset $\gamma\subseteq\Gamma$ induces a partition $V=V_{\gamma}\cup V_{\gamma^\prime}$, where
%$$
%V_{\gamma}=\{\bm{x}_i\in V\,|\, \bm{x}_i\in\gamma\}
%\quad\mbox{and}\quad 
%V_{\gamma^\prime}=\{\bm{x}_i\in V\,|\, \bm{x}_i\in\gamma^\prime\} = V\setminus V_{\gamma},
%$$
%respectively. Note that $V_{\Gamma^\prime}$ is simply the set of all interior vertices $V_0$. 
%
If a vertex $\bm{v}_i\in\Gamma$ then the boundary of its dual cell $\omega_i$ has a non-empty intersection with $\Gamma$, which we denote as $\Gamma_i := \partial\omega_i\cap \Gamma$. As in Section \ref{sec:notation-general} we shall always assume that $\gamma$ and $\gamma^\prime$ are  exact unions of such segments.
It follows that
\begin{equation}\label{eq:cell_bdry}
\partial\omega_i 
%=\Big(\bigcup_{j\in V_i} \bm{f}_{ij}\Big) \cup \chi_i(\Gamma)
=\Big(\bigcup_{j\in V_i} \bm{f}_{ij}\Big) \cup \chi_i(\gamma^\prime) \cup \chi_i(\gamma) \,.
\end{equation}
%
%Restricting  \eqref{eq:d} to dual cells and taking in account \eqref{eq:cell_bdry} yields a mimetic divergence operator $DIV : \widetilde{U}^h_\gamma \rightarrow \widetilde{S}^h$ given by
Restriction of \eqref{eq:d} to dual cells $\omega_i$, followed by approximation of the integrals on $\partial\omega_i$ by single point quadrature yields a discrete divergence operator $DIV : \widetilde{U}^h_\gamma \rightarrow \widetilde{S}^h$, whose action on 
$\bm{u}^h\in \widetilde{U}^h_\gamma$ is given by   
\begin{equation}\label{eq:mimetic_div}
(DIV \bm{u}^h)_i :=
\frac{1}{\mu_i}\left[
\sum_{\bm{f}_{ij}\in\partial\omega_i} u_{ij} \mu_{ij} +
u_{\chi_i(\gamma^\prime)} \mu _{\chi_i(\gamma^\prime)} +
\int_{\chi_i(\gamma)}\bm{u}\cdot\bm{n} dS
\right] \qquad\forall\omega_i\in C .
\end{equation}
In this operator
%All three key pieces of information necessary for  \eqref{eq:mimetic_div} are readily available through the discretization infrastructure assumed in this section. In particular, 
%
\begin{itemize}
\setlength\itemsep{0.3ex}
\item the \emph{field data} is provided by the degrees-of-freedom in $\widetilde{U}^h_\gamma$, i.e., the sets of numbers $\{u_{ij}\}$ and $\{u_{\Gamma_i}\}$, such that $u_{ij} = u_{ji}$;
\item the \emph{metric data} is provided by the cell volumes and the oriented face areas in the dual mesh, i.e., the sets of numbers $\{\mu_i\}$, $\{\mu_{ij}\}$ and $\{\mu_{\Gamma_i}\}$, respectively, such that $\mu_{ij} = -\mu_{ji}$; 
\item the \emph{topological data} is provided by formula \eqref{eq:cell_bdry}. 
\end{itemize}

\begin{remark}\label{rem:field_data}
In mixed discretizations \cite{Brezzi_91_BOOK} of, e.g., diffusion problems , $\{u_{ij}\}$ and $\{u_{\Gamma_i}\}$ are degrees-of-freedom describing an $H(div)$-conforming discretization of a flux field. In contrast, in finite volume discretizations of conservation laws \cite{Leveque_02_BOOK} the sets $\{u_{ij}\}$ and $\{u_{\Gamma_i}\}$ represent values derived from another set of degrees-of-freedom, usually describing a scalar dependent variable such as a pressure field.
\end{remark}

One can show that \eqref{eq:mimetic_div} is first-order accurate, i.e., for all sufficiently regular vector fields $\bm{u}$ there holds
\begin{equation}\label{eq:mimetic_div_accuracy}
(DIV \mathcal{R}(\bm{u}))_i = \nabla\cdot\bm{u}(\bm{x}_i) + O(h) \,.
\end{equation}
The following result confirms that \eqref{eq:mimetic_div} satisifes a discrete divergence theorem, i.e., that it is compatible. 

 %%%%%%%%
\begin{lemma}\label{lem:mimetic_property}
Consider a collection $C_\omega$ of dual cell indices and define  $\omega\subseteq\Omega$ as
$$
\omega = \bigcup_{i \in C_\omega} \omega_i  \,.
$$
Let $\partial\omega_0 = \partial \omega\cap\Omega$, $\partial\omega_\gamma = \partial\omega\cap\gamma$, and  $\partial\omega_{\gamma^\prime} = \partial\omega\cap\gamma^\prime$. Then, 
\begin{equation}\label{eq:mimetic_property}
\sum_{i \in C_\omega} \mu_i (DIV \bm{u}^h)_i 
=
\sum_{\bm{f}_{ij} \in \partial\omega_0}  {u}_{ij}{\mu}_{ij} +
\sum_{\Gamma_k\in \partial\omega_{\gamma^\prime}} u_{\Gamma_k} \mu _{\Gamma_k} +
\sum_{\Gamma_k\in \partial\omega_\gamma} \int_{\Gamma_k} \bm{u}\cdot\bm{n}_k dS \,.
\end{equation}
\end{lemma}
\begin{proof}
Summing \eqref{eq:mimetic_div} over the cells in $\omega$ gives
$$
\begin{array}{rcl}
\displaystyle
\sum_{i \in C_\omega} \mu_i (DIV \bm{u}^h)_i 
&=&\displaystyle
\sum_{i \in C_\omega}
\left[
\sum_{\bm{f}_{ij}\in\partial\omega_i} u_{ij} \mu_{ij} +
u_{\chi_i(\gamma^\prime)} \mu _{\chi_i(\gamma^\prime)} +
\int_{\chi_i(\gamma)}\bm{u}\cdot\bm{n} dS
\right] \\[3ex]
\qquad
&=&\displaystyle
\sum_{i \in C_\omega}
\sum_{\bm{f}_{ij}\in\partial\omega_i} u_{ij} \mu_{ij} +
\sum_{\Gamma_k\in \partial\omega_{\gamma^\prime}} u_{\Gamma_k} \mu _{\Gamma_k}+
\sum_{\Gamma_k\in \partial\omega_\gamma} \int_{\Gamma_k} \bm{u}\cdot\bm{n}_k dS \,.
\end{array}
$$
Since $\mu_{ij}=-\mu_{ji}$ all terms in the double sum above corresponding to pairs of adjacent cells cancel each other, leaving only the terms for which $\bm{f}_{ij}\in \partial\omega_0$. This completes the proof. 
\end{proof}
Global conservation in $\Omega$ a direct consequence of Lemma \ref{lem:mimetic_property}. For simplicity we state it with $\gamma=\Gamma$. 

\begin{corollary}\label{cor:mimetic_property_global}
Assume that $\bm{u}^h\in \widetilde{U}^h_\Gamma$ and let  $\bm{u}\in H_\Gamma(div,\Omega)$ be given. Then,
\begin{equation}\label{eq:mimetic_property_global}
\sum_{\omega_i\in C} \mu_i (DIV \bm{u}^h)_i 
=\int_{\Gamma} \bm{u}\cdot\bm{n}\, dS \,.
\end{equation}
\end{corollary}
\begin{proof}
The result follows by setting $\omega=\Omega$ in Lemma \ref{lem:mimetic_property} and noting that 
$\partial\omega_0 = \emptyset$, $\partial\omega_{\gamma^\prime}=\emptyset$, and $\partial\omega_{\gamma}=\Gamma$. 
\end{proof}
%
%Note that  \eqref{eq:mimetic_property_global} together with the divergence theorem imply that
%\begin{equation}\label{eq:identity1}
%\sum_{\omega_i \in C_0} \mu_i (DIV \bm{u}^h)_i = \int_{\Omega} \nabla\cdot\bm{u} dV \,.
%\end{equation}
%Of course, this result assumes that the boundary integrals in \eqref{eq:mimetic_div} are computed exactly. 
%In most cases these integrals have to be approximated and so in practice \eqref{eq:identity1} can only hold up to a quadrature error. 

%%%%%%%%%%%%%%%%%%%%%%%
\section{An abstract meshfree mimetic divergence operator}\label{sec:abstract}
%%%%%%%%%%%%%%%%%%%%%%%
We consider a meshfree discretization infrastructure comprising
\begin{itemize}\setlength\itemsep{0.3ex}
\item a point cloud $X=X_0\cup X_\Gamma$, representing the computational domain $\Omega$, and 
\item discrete spaces $S^h$ and $U^h_\gamma$, representing scalar and vector fields by their point samples on the cloud.
\end{itemize}
Our goal is to endow this infrastructure with an abstract meshfree mimetic divergence (MMD) operator whose properties mirror those of its mesh-based prototype \eqref{eq:mimetic_div}. Thus, we seek a mapping ${DIV}: U^h_\gamma \rightarrow S^h $ that is  first-order accurate and satisfies a discrete divergence theorem. 
We shall define this mapping by mimicking the action of \eqref{eq:mimetic_div} which assigns to every dual cell $\omega_i$ a \emph{contraction} of the field and metric data living on the faces in $\partial\omega_i$, and weighted by the reciprocal of the cell's measure. 
This task requires appropriate abstractions for the field and metric data, as well as a suitable notion of a boundary operator. 

We shall define this  operator through a \emph{virtual} primal-dual mesh complex. To reduce notational clutter we  reuse the nomenclature from Section \ref{sec:mimeticdiv}.  
We call a collection of integer indices $V=\{i\}_{i=1}^{p_V}$ a \emph{virtual} vertex set and refer to its elements as $\bm{v}_i$. 
%Given a bounded region $\Omega$, the set $V$ induces a virtual length scale given by 
%$$
%h_V = \left(\frac{| \Omega |}{ p_v }\right)^{\frac1d} \,.
%$$
A virtual edge is simply an ordered pair of virtual vertices, i.e., 
$$
\bm{e}_{ij} =\left\{ (ij) \,\big|\, (ij)\in V\times V \right\} \,.
$$ 
A collection of $p_E$ virtual edges forms a set $E$, which together with $V$ defines the virtual primal mesh $\tau_h(V,E)$. 
To every $\bm{v}_i \in V$  and $\bm{e}_{ij}\in E$ we assign a virtual dual cell $\omega_i$ and a virtual dual face $\bm{f}_{ij} $, respectively. The sets of these dual entities are denoted by $C$ and $F$, respectively, and they form a virtual topologically dual mesh $\tau^\prime_h(C,F)$. 
A practically useful virtual boundary operator $\partial C\rightarrow F$ should be able to recover the physical domain boundary $\Gamma$. This requires some degree of association between the virtual mesh structures and the  physical domain $\Omega$. 
A simple and effective way to establish such an association is to tie every virtual vertex to a point in the cloud $X$. This induces a partition of $V$ into \emph{interior} vertices $V_0=X_0$ and boundary vertices $V_\Gamma=X_\Gamma$, and prompts the following definition:
\begin{equation}\label{eq:virtual_bdry}
\partial\omega_i =
\left\{
\begin{array}{ll}
\{\bm{f}_{ij}\}, \ j\in V_i & \mbox{if $i\in V_0$} \\[2ex]
\{\bm{f}_{ij}\}\cup \Gamma_i, \ j\in V_i & \mbox{if $i\in V_\Gamma$} 
\end{array}
\right.; \quad\forall\omega_i\in C.
\end{equation}
This definition allows us to write the virtual cell boundary in the same form as in \eqref{eq:cell_bdry}, i.e.,
\begin{equation}\label{eq:cell_bdry_virt}
\partial\omega_i 
=\Big(\bigcup_{j\in V_i} \bm{f}_{ij}\Big) \cup \chi_i(\gamma^\prime) \cup \chi_i(\gamma) \,.
\end{equation}
Note that while $\omega_i$ is a purely virtual entity, its boundary may contain ``real'' geometric entities from $\Gamma$.

Consider next the field data abstraction. To mimic \eqref{eq:mimetic_div} the domain of the MMD operator should be associated with the virtual faces $F$, whereas its range should contain fields living on the virtual cells $C$. 
To ensure first-order accuracy the latter should be piecewise constant with respect to $C$. Thus, we can reuse the definition of $\widetilde{S}^h$ from  \eqref{eq:dual_fields} with the understanding that each degree-of-freedom is now assigned to a virtual, rather than to a physical, dual cell. 
Likewise, the domain space should be able to reproduce exactly the preimage of $\widetilde{S}^h$, i.e., any vector field $\bm{u}$ such that $\nabla\cdot\bm{u}\in \widetilde{S}^h$. In the mesh-based case this can be accomplished by a single\footnote{For example, the lowest-order Raviart-Thomas elements on $d$-simplices \cite{Raviart_77_INPROC}, also known as Whitney 2-forms \cite{Bossavit_88_IEE},  are incomplete linear polynomial fields of the form $\bm{a} + b\bm{x}$ which have $d+1$ degrees-of-freedom.} degree-of-freedom per face. 
However, this may not be enough in the meshfree setting and so we shall allow for $n_F \ge 1$ \emph{field moments} per virtual face $\bm{f}_{ij}\in F$ and boundary segment $\Gamma_i\in\gamma$. 
%However, this may not be enough in the meshfree setting and so instead here we shall allow for multiple degrees-of-freedom per virtual face. Thus, to every $\bm{f}_{ij}\in F$ and $\Gamma_i\in\gamma$ we shall assign $n_F \ge 1$ \emph{field moments}. 
%
In sum, we define the domain $\widetilde{U}^h_\gamma$ and the range $\widetilde{S}^h$ for the abstract MMD operator by modifying \eqref{eq:dual_fields} to 
\begin{equation}\label{eq:dual_fields_virtual}
 \bm{u}^h =\{ \bm{u}_{ij}\in\mathbb{R}^{n_F}\,|\, \bm{f}_{ij}\in F\} \cup 
\{\bm{u}_{\Gamma_i}\in\mathbb{R}^{n_F}\,|\, \Gamma_i\in \gamma^\prime\}
\quad\mbox{and}\quad
s^h =\{ \phi_i \in \mathbb{R}\,|\, \omega_i\in C\},
\end{equation}
respectively, with the understanding that all dual mesh entities in \eqref{eq:dual_fields_virtual} are virtual. 

Finally, let us consider the abstraction for the metric data.  We seek this abstraction in the form of virtual face and cell \emph{moments}, i.e., sets of numbers assigned to every element of $F$ and $C$. To allow contraction of these moments with the field moments in \eqref{eq:dual_fields_virtual} we chose the former to be the algebraic duals of the latter.
Thus,  we assign a real number $\mu_i$ to every virtual dual cell $\omega_i\in C$, a real vector $\bm{\mu}_{ij}\in\mathbb{R}^{n_F}$ to every virtual dual face $\bm{f}_{ij}\in F$ and a real vector $\bm{\mu}_{\Gamma_i}\in\mathbb{R}^{n_F}$ to every boundary segment $\Gamma_i\in\gamma^\prime$. 
Following \eqref{eq:mimetic_div} we then define an abstract operator $DIV: \widetilde{U}^h_\gamma \rightarrow \widetilde{S}^h $ by contracting the metric and the field data:
\begin{equation}\label{eq:mimetic_div_abs_prelim}
(DIV \bm{u}^h)_i :=
\frac{1}{\mu_i}\left[
\sum_{\bm{f}_{ij}\in\partial\omega_i} \bm{u}_{ij} \cdot \bm{\mu}_{ij} +
\bm{u}_{\chi_i(\gamma^\prime)} \cdot \bm{\mu}_{\chi_i(\gamma^\prime)} +
\int_{\chi_i(\gamma)}\bm{u}\cdot\bm{n} dS
\right] \qquad\forall\omega_i\in C .
\end{equation}
This operator acts on the \emph{field moments}  $\{\bm{u}_{ij}\}$ and $\{\bm{u}_{\Gamma_i}\}$ rather than on the point samples $U^h_\gamma$ that are the ``native'' representation of vector fields on the point cloud. Thus, deployment of \eqref{eq:mimetic_div_abs_prelim} in our meshfree infrastructure requires  a particle-to-virtual face data transfer operator\footnote{Formally we also need a data transfer operator $T_F: \widetilde{S}^h \rightarrow S^h$ to map the discrete divergence back to the cloud points. However, since in the present context  $S^h\equiv \widetilde{S}^h$, this operator is the identity and is omitted for simplicity.}
\begin{equation}\label{eq:point-mesh}
T_{F} : U^h_\gamma \rightarrow \widetilde{U}^h_\gamma;\quad
(T_{F}\bm{u}^h)\big |_{\bm{f}} = \bm{t}_{\bm{f}}(\bm{u}^h)\in \mathbb{R}^{n_F} \,,
\bm{f}\in\{\bm{f}_{ij}\}\cup \{\Gamma_j \,|\, \bm{x}_j\in\gamma^\prime\}\,.
\end{equation}
which maps point samples  into  field moments. Combining the action of this operator with \eqref{eq:mimetic_div_abs_prelim} yields the desired abstract MMD operator $DIV: U^h_\gamma \rightarrow S^h $, where
\begin{equation}\label{eq:mimetic_div_abs}
(DIV \bm{u}^h)_i:=
\frac{1}{\mu_i}\left[
\sum_{\bm{f}_{ij}\in\partial\omega_i} \bm{t}_{ij}(\bm{u}^h) \cdot \bm{\mu}_{ij} +
\bm{t}_{\chi_i(\gamma^\prime)}(\bm{u}^h) \cdot \bm{\mu}_{\chi_i(\gamma^\prime)} +
\int_{\chi_i(\gamma)}\bm{u}\cdot\bm{n} dS
\right] \qquad\forall\omega_i\in C .
\end{equation}
The following section studies the  properties of \eqref{eq:mimetic_div_abs}.

%\begin{remark}\label{rem:scale}
%The association of $V$ with the point cloud implies that the virtual length scale $h_V$ is equivalent to the fill distance  $h_X$ of the point cloud. Thus, without loss of generality from now on we shall assume that $h_V = h_X$. 
%\end{remark}

%%%%%%%%%%
\subsection{Analysis of the abstract MMD operator}\label{sec:theory-abstract-MMD}
%%%%%%%%%%
%
To ensure that \eqref{eq:mimetic_div_abs} has the same mimetic properties as its mesh-based prototype we require that 
\begin{description}\setlength\itemsep{0.1ex}
\item[T.1] The virtual face volumes satisfy $\mu_i>0$, $\mu_i=O(h^{d}_{X})$, and $\sum_i \mu_i = \mu(\Omega)$.
\item[T.2] The virtual face moments $\{\bm{\mu}_{ij}\}$ are antisymmetric: 
$\bm{\mu}_{ij} = - \bm{\mu}_{ji}$.
\item[T.3] The data transfer operator \eqref{eq:point-mesh} is symmetric:
$
\bm{t}_{ij}(\bm{u}^h) = \bm{t}_{ji}(\bm{u}^h) .
$
\end{description}
\textbf{T.1}  ensures that the virtual volumes behave like physical volumes.
The ``topological'' requirements \textbf{T.2}-\textbf{T.3} are prompted by Lemma \ref{lem:mimetic_property}, which reveals that mimetic properties of \eqref{eq:mimetic_div} hinge on the symmetry of the face degrees-of-freedom $u_{ij}$ and the antisymmetry of the oriented face areas $\mu_{ij}$. The following result shows that \textbf{T.1}-\textbf{T.3} are indeed sufficient for the abstract MMD operator to be locally and globally conservative. 

\begin{theorem}\label{thm:mimetic_div_abs_cons}
Assume that \textbf{T.1}-\textbf{T.3} hold. Then the abstract MMD operator \eqref{eq:mimetic_div_abs} satisfies \eqref{eq:mimetic_property} (local conservation) and \eqref{eq:mimetic_property_global} (global conservation).
\end{theorem}
\begin{proof}
Owing to \textbf{T.2} and \textbf{T.3} the proof of Lemma \ref{lem:mimetic_property} applies verbatim to the abstract MMD operator as well. The global conservation \eqref{eq:mimetic_property_global} then follows from \textbf{T.1} and \eqref{eq:mimetic_property} as in Corollary \ref{cor:mimetic_property_global}.
\end{proof}

While conservation is a purely ``topological'' property that can be secured by independent conditions on the field and face moments, this is not the case with the consistency of the MMD operator, which requires these moments to work well together. The following two definitions introduce conditions that coordinate the properties of the metric and the field data to ensure the desired first-order accuracy.
\begin{definition} \label{def:consistent}
A set of virtual cell volumes and virtual face moments  $\{\mu_i,\bm{\mu}_{ij}, \bm{\mu}_{\Gamma_i}\}$ and a data transfer operator $T_F:U^h_\gamma \rightarrow \widetilde{U}^h_\gamma$ are called a $P_1$-reproducing pair  if for every $\bm{p}\in (P_1)^d$ and its point sample $\bm{p}^h\in U^h_\gamma$ there holds
\begin{equation}\label{eq:mimetic_div_P1}
\nabla\cdot\bm{p} \,\big |_{\bm{x}_i} = 
\frac{1}{\mu_i}\left[
\sum_{\bm{f}_{ij}\in\partial\omega_i} \bm{t}_{ij}(\bm{p}^h) \cdot \bm{\mu}_{ij} +
\bm{t}_{\chi_i(\gamma^\prime)}(\bm{p}^h) \cdot \bm{\mu}_{\chi_i(\gamma^\prime)} +
\int_{\chi_i(\gamma)}\bm{p}\cdot\bm{n} dS
\right] \qquad\forall\omega_i\in C .
\end{equation}
\end{definition}
In other words, $\{T_F; \{\mu_i,\bm{\mu}_{ij}, \bm{\mu}_{\Gamma_i}\}\}$ is a $P_1$-reproducing pair iff for every 
$\bm{p}\in (P_1)^d$ there holds
$$
(DIV \bm{p}^h)_i = \nabla\cdot\bm{p} \,\big |_{\bm{x}_i} 
\quad\forall \bm{x}_i\in X_0\cup X_{\gamma^\prime} .
$$

\begin{definition}\label{def:locLip}
A set of virtual face moments $\{\bm{\mu}_{ij}, \bm{\mu}_{\Gamma_i}\}$ and a data transfer operator $T_F:U^h_\gamma \rightarrow \widetilde{U}^h_\gamma$ are called a locally Lipschitz-continuous pair on $X$ if there exist balls $\{B_{\bm{f}}^h\}$ of radius $O(h_X)$, centered at the barycenters of the virtual faces $\bm{f}\in F$ and a constant $C>0$, such that for every $\bm{u},\bm{v}\in (C^2(\Omega))^d$ with point samples $\bm{u}^h,\bm{v}^h\in U^h_\gamma$ there holds
\begin{equation}\label{eq:locLip}
| \bm{t}_{\bm{f}}(\bm{u}^h)\cdot \bm{\mu}_{\bm{f}} - \bm{t}_{\bm{f}}(\bm{v}^h)\cdot \bm{\mu}_{\bm{f}} | 
\le C h_{X}^{d-1} \| \bm{u}^h - \bm{v}^h \|_{\ell^\infty, X_{B_{\bm{f}}^h}}\,,
\quad
\forall \bm{f}\in\{\bm{f}_{ij}\}\cup \{\Gamma_j \,|\, \bm{x}_j\in\gamma^\prime\}\,.
\end{equation}
\end{definition}

These two properties are sufficient for the abstract MMD operator to be first-order accurate. 

\begin{theorem}\label{thm:error}
Assume that $\{T_F; \{\mu_i, \bm{\mu}_{ij}, \bm{\mu}_{\Gamma_i}\}\}$ is a  $P_1$-reproducing pair, $\{T_F,\{\bm{\mu}_{ij}, \bm{\mu}_{\Gamma_i}\}\}$ is a locally Lipschitz continuous pair on $X$, and that  \textbf{T.1} holds. 
Then, there exists a constant $C$, independent of the point cloud and such that for every $\bm{u}\in (C^2(\Omega))^d$ and its point sample $\bm{u}^h\in U^h_\gamma$ there holds the error bound
\begin{equation}\label{eq:MMD_abs_error}
\left \| \nabla\cdot\bm{u} - (DIV \bm{u}^h)\right \|_{\ell^\infty,X} 
\le Ch \| \bm{u}\|_{C^2(\Omega)} .
\end{equation}
\end{theorem}
\begin{proof}
Consider a linear vector field $\bm{p}\in (P_1)^d$ with a point sample $\bm{p}^h\in U^h_\gamma$, and an arbitrary point $\bm{x}_i\in X$. From the triangle inequality and the $P_1$-reproducing pair assumption it follows that
$$
\begin{aligned}
\left | \nabla\cdot\bm{u}(\bm{x}_i) - (DIV \bm{u}^h)_i \right | 
&\le   
\left | \nabla\cdot\bm{u}(\bm{x}_i) - \nabla\cdot\bm{p}(\bm{x}_i)\right | +
\left | \nabla\cdot\bm{p}(\bm{x}_i) - (DIV \bm{u}^h)_i \right | \\
&=
\left | \nabla\cdot\bm{u}(\bm{x}_i) - \nabla\cdot\bm{p}(\bm{x}_i)\right | +
\left | (DIV \bm{p}^h)_i  - (DIV \bm{u}^h)_i \right | 
\end{aligned}
$$
Consider first the case $\chi_i(\gamma^\prime)=\emptyset$. Using \eqref{eq:locLip} then yields the following bound for the second term:
$$
\begin{array}{l}
\displaystyle
\left | (DIV \bm{p}^h)_i  - (DIV \bm{u}^h)_i \right | 
%\\[1ex]
%\quad 
\le 
\displaystyle
\frac{1}{\mu_i}\left[
\sum_{\bm{f}_{ij}\in\partial\omega_i} 
| \bm{t}_{ij}(\bm{p}^h)\cdot \bm{\mu}_{ij}- \bm{t}_{ij}(\bm{u}^h)\cdot \bm{\mu}_{ij} | 
%+|\bm{t}^\intercal_{\chi_i(\gamma^\prime)}(\bm{p}^h) - \bm{t}^\intercal_{\chi_i(\gamma^\prime)}(\bm{u}^h) | \, |\bm{\mu}_{\chi_i(\gamma^\prime)}| 
+
\int_{\chi_i(\gamma)}|\bm{p}-\bm{u} | \,dS
\right] \\[4ex]
\quad\le
\displaystyle
\frac1\mu_i\big[
C h^{d-1}_{X}\|\bm{p}^h - \bm{u}^h \|_{\ell^\infty, X_{B^h_{ij}}} 
+ \mu(\chi_i(\gamma))\| \bm{p} - \bm{u} \|_{C^0(\chi_i(\gamma))} \big]
\end{array}
$$
where $\mu(\chi_i(\gamma))$ is the measure of  $\chi_i(\gamma)$. Assumption \textbf{T.1} implies that  $\mu_i = O(h_X^d)$, while $\mu(\chi_i(\gamma))$ is either zero or $O(h_X^{d-1})$. As a result, we have that 
$$
\begin{aligned}
\displaystyle
\left | (DIV \bm{p}^h)_i  - (DIV \bm{u}^h)_i \right |  
&\le 
\displaystyle
C h^{-1}_X 
\left[ 
\|\bm{p}^h - \bm{u}^h \|_{\ell^\infty, X_{B^h_{ij}}}  + 
\| \bm{p} - \bm{u} \|_{C^0(\chi_i(\gamma))}\right]
\\[0.5ex]
&\le\displaystyle
C h^{-1}_X 
\left[
\|\bm{p} - \bm{u}\|_{C^0(B^h_{ij})} + \| \bm{p} - \bm{u} \|_{C^0(\chi_i(\gamma))} 
\right] 
\end{aligned}
$$ 
Taking  $\bm{p}$ to be the linear Taylor polynomial of $\bm{u}$ at the point $\bm{x}_i$ then yields the bound
$$
\|\bm{p} - \bm{u}\|_{C^0(B^h_{ij})} +
\| \bm{p} - \bm{u} \|_{C^0(\chi_i(\gamma))} 
\le Ch^2_X \|\bm{u}\|_{C^2(\Omega)}
$$
for the second term and the bound
$$
\left | \nabla\cdot\bm{u}(\bm{x}_i) - \nabla\cdot\bm{p}(\bm{x}_i)\right |
\le 
Ch_X \|\bm{u}\|_{C^2(\Omega)}
$$
for the first term. This completes the proof for $\chi_i(\gamma^\prime)=\emptyset$. 
When $\chi_i(\gamma^\prime)=\Gamma_i$ the extra term $\|\bm{p}^h - \bm{u}^h \|_{\ell^\infty, X_{B^h_{i}}} $ can be estimated by $\|\bm{p} - \bm{u}\|_{C^0(B^h_{i})}$. The proof follows by combining this and the above bounds. 
\end{proof}

%%%%%%%%%%%%%%%%%%%%%%%
\section{Meshfree mimetic divergence (MMD) instantiations}\label{sec:MMD-instances}
%%%%%%%%%%%%%%%%%%%%%%%

In this section we present two instances of the abstract MMD operator. 
Section \ref{sec:gmlsquad} implements \eqref{eq:mimetic_div_abs} using a background mesh. This \emph{hybrid} formulation is useful for applications where such a grid is available but its quality is poor. For example, the mesh may contain valid, but almost degenerate, ``sliver'' elements that render the mesh-based shape functions nearly singular. 
In this case, a meshfree approach such as \eqref{eq:mimetic_div_abs}  can provide a more robust and accurate alternative to traditional mesh-based discretizations by using the mesh only to supply the metric and topological information but not to define the field data and the operators acting on it. 
Section  \ref{sec:divtheorem} shows how to implement \eqref{eq:mimetic_div_abs} without a background mesh. This MMD instance targets applications where mesh generation creates a computational bottleneck and should be avoided.

%%%%%%%%%%%%%%%%%%%%%%%
\subsection{MMD with a background mesh}\label{sec:gmlsquad}
%%%%%%%%%%%%%%%%%%%%%%%
We assume that the virtual primal-dual mesh complex can be associated with a conforming physical grid on $\Omega$.  
%
%This setting is relevant to applications where grid is available but its quality is poor. For example, the mesh may contain valid, but almost degenerate, ``sliver'' elements that render the mesh-based shape functions nearly singular. 
%%
%In this case, a meshfree approach such as \eqref{eq:mimetic_div_abs}  can provide a more robust and accurate alternative to traditional mesh-based discretizations by using the mesh only to supply the metric and topological information but not to define the field data and the operators acting on it. 
%
Existence of a physical dual mesh implies that a set  $\{\mu_i\}$ satisfying \textbf{T.1} is readily available. 
Thus, to instantiate \eqref{eq:mimetic_div_abs} it remains to find a data transfer operator $T_F$ and virtual face moments $\{\bm{\mu}_{ij},\bm{\mu}_{\Gamma_i}\}$ such that assumptions \textbf{T.2}--\textbf{T.3} hold,  $\{T_F; \{\mu_i, \bm{\mu}_{ij}, \bm{\mu}_{\Gamma_i}\}\}$ is a  $P_1$-reproducing pair, and  $\{T_F,\{\bm{\mu}_{ij}, \bm{\mu}_{\Gamma_i}\}\}$ is a locally Lipschitz continuous pair on $X$. 
To that end we shall use GMLS with targets  
\begin{equation}\label{eq:target}
\tau_{\bm{f}}(\bm{u}) = \int_{\bm{f}} \bm{u}\cdot\bm{n}_{\bm{f}}  d S \,,
%\quad
%\bm{f}\in F \cup \{\Gamma_i \,|\, \bm{x}_i\in\gamma^\prime\} \,,
\end{equation} 
where $\bm{f}$ is a dual face or a segment on $\gamma^\prime$, i.e., $\bm{f}\in F \cup \{\Gamma_i \,|\, \bm{x}_i\in\gamma^\prime\}$. 
We apply GMLS with
\begin{itemize}\setlength\itemsep{0.ex}
\item a function space $U = (C^1(\Omega))^d$; 
\item a reproduction space $\Phi=(P_1)^d$ with dimension $n_{1} = d(d+1)$ and basis $\{\bm{\phi}_s\}_{s=1}^{n_{1}}$ with basis functions
\begin{equation}\label{eq:P1V_basis}
\bm{\phi}_{s}(\bm{x}) = \mathbf{e}_k\, \phi_r(\bm{x});\quad
s=(k-1)(d+1)+r; \quad
k=1,\ldots,d;\  r=1,\ldots, d+1\,,
\end{equation}
where $\mathbf{e}_k$ is the $k$th canonical basis vector in $\mathbb{R}^d$ and $\{\phi_r(\bm{x})\}_{r=1}^{d+1}$ is a basis for  $P_1$;
\item sampling functionals\footnote{We refer to \cite[Chapter 3]{WendlandBook} for conditions on the sampling locations $\{\bm{x}_i\}$ that are sufficient to ensure polynomial reproduction and rigorous approximation error estimates. } $\lambda_i = \delta_{\bm{x}_i}$ mapping vector fields $\bm{u}\in U$ into point samples $\bm{u}^h\in U^h_\gamma$ on the cloud;
\item kernels $w(\tau_{\bm{f}};\lambda_i) = \kappa(|\bm{x}_{\bm{f}} - \bm{x}_i |)$ where $\bm{x}_{\bm{f}}$ is the centroid of  $\bm{f}\in F \cup \{\Gamma_j \,|\, \bm{x}_j\in\gamma^\prime\}$ and $\kappa(\rho) = \left(1-{\rho}/{\epsilon_\kappa}\right)^4_+$, with $\epsilon_\kappa$ is large enough to obtain $P_1$-unisolvency. The face centroids $\{\bm{x}_{\bm f}\}$ are not required to coincide with the midpoints $\{\bm{x}_{ij}\}$ of the corresponding primal edges.
\end{itemize}
With these choices the GMLS approximation of \eqref{eq:target} is exact for linear polynomials and is given by 
\begin{equation}\label{eq:GMLSface}
\tau_{\bm{f}}(\bm{u}) \approx \widetilde{\tau}_{\bm{f}}(\bm{u}^h) 
:= \bm{c}_{\bm{f}}(\bm{u}^h)\cdot \bm{\tau}_{\bm{f}}(\bm{\phi})\,;
\quad\forall
\bm{f}\in F \cup \{\Gamma_j \,|\, \bm{x}_j\in\gamma^\prime\} \,.
\end{equation}
The GMLS coefficients $\bm{c}_{\bm{f}}(\bm{u}^h)\in\mathbb{R}^{n_{1}}$ solve the local weighted least-squares problems 
\begin{equation}\label{eq:gmls-P1}
\bm{c}_{\bm{f}}(\bm{u}^h)
= \argmin_{\bm{b}\in\mathbb{R}^{n_{1}}}\frac12
\left | B\bm{b} -  \bm{u}^h  \right |^2_{W(\tau_{\bm{f}})} \,,
\end{equation}
which transform point samples into ``field moments'' and define a mapping
\begin{equation}\label{eq:point-mesh-GMLS}
T_{F} : U^h_\gamma \rightarrow \widetilde{U}^h_\gamma;\quad
(T_{F}\bm{u}^h)\big |_{\bm{f}} := \bm{c}_{\bm{f}}(\bm{u}^h)\in \mathbb{R}^{n_1}, \ \ 
\forall \bm{f}\in F \cup \{\Gamma_j \,|\, \bm{x}_j\in\gamma^\prime\}\,.
\end{equation}
We chose \eqref{eq:point-mesh-GMLS} as the data transfer operator for  \eqref{eq:mimetic_div_abs}. 
The second term in  \eqref{eq:GMLSface}, i.e., the vector $\bm{\tau}_{\bm{f}}(\bm{\phi})\in \mathbb{R}^{n_{1}}$, contains the integrals of the polynomial basis functions on $\bm{f}$ and defines the \emph{face moments} for \eqref{eq:mimetic_div_abs}. Setting $\bm{\mu}_{ij}= \bm{\tau}_{ij}(\bm{\phi})$ and $\bm{\mu}_{\Gamma_j}= \bm{\tau}_{\Gamma_j}(\phi)$ in \eqref{eq:mimetic_div_abs} we obtain the following \emph{GMLS instance} of the MMD operator:
\begin{equation}\label{eq:div-dis-GMLS}
%(\overset{{\mbox{\tiny\sf GMLS}}}{DIV} 
(\DIVGMLS\bm{u}^h)_i  
= 
\frac{1}{\mu_i} \left[
\sum_{\bm{f}_{ij}\in\partial\omega_i}(\bm{c}_{ij}(\bm{u}^h)) \cdot \bm{\mu}_{ij} +
\bm{c}_{\chi_i(\gamma^\prime)}(\bm{u}^h) \cdot \bm{\mu}_{\chi_i(\gamma^\prime)} +
\int_{\chi_i(\gamma)}\bm{u}\cdot\bm{n} dS
\right] \qquad\forall\omega_i\in C .
\end{equation}
%

%%%%%%%
\subsubsection{Analysis}\label{sec:MMD-GMLS-analysis}
%%%%%%%
Since $\mu_i$ are volumes of physical dual cells on a conforming mesh, assumption \textbf{T.1} readily holds. We now check the rest of the assumptions of the abstract theory in \S \ref{sec:theory-abstract-MMD}. To avoid non-essential technicalities we shall assume that all dual faces are affine.  

%%%%%%%
\begin{lemma}\label{lem:GMLS-P1}
The pair $\{T_F, \{\mu_i,\bm{\mu}_{ij},\bm{\mu}_{\Gamma_i}\}\}$ is $P_1$-reproducing.
\end{lemma}
\begin{proof}
Let $\bm{p}\in (P_1)^d$ and $\bm{p}^h\in U^h_\gamma$ denote its point sample on the cloud. By construction, the GMLS approximation \eqref{eq:GMLSface} is exact for linear polynomials and so,
$$
\bm{c}_{ij}(\bm{p}^h) \cdot \bm{\mu}_{ij} 
=
\int_{\bm{f}_{ij}} \bm{p}\cdot\bm{n}_{ij} \,dS 
\quad\mbox{and}\quad
\bm{c}_{j}(\bm{p}^h)\cdot \bm{\mu}_{\Gamma_j}
=
\int_{\Gamma_{j}} \bm{p}\cdot\bm{n}_{j} \,dS \,.
$$
Inserting these identities in \eqref{eq:div-dis-GMLS}, using  the divergence theorem and the fact that $\nabla\cdot\bm{p}$ is constant, yields 
$$
(\DIVGMLS \bm{p}^h)_i 
=
\frac1\mu_i\int_{\partial\omega_i} \bm{p}\cdot\bm{n}_i dS
=
\frac1\mu_i\int_{\omega_i} \nabla\cdot\bm{p} \,dx
=
\nabla\cdot\bm{p} \,\big |_{\bm{x}_i}\,.
$$
%This completes the proof.
\end{proof}

%%%%%%%
\begin{lemma}\label{lem:GMLS-sym}
The face moments $\{\bm{\mu}_{ij}\}$ and the operator $T_F$ satisfy assumptions \textbf{T.2} and \textbf{T.3}, respectively.
\end{lemma}
%%%%%%%
\begin{proof}
To show that \textbf{T.3} holds note that the kernel $w(\tau_{ij};\lambda_i)$ is anchored at the centroid $\bm{x}_{\bm{f}_{ij}}$ of $\bm{f}_{ij}$. As a result, \eqref{gmlscoeffs} implies that $\bm{c}_{ij}(\bm{u}^h)=\bm{c}_{ji}(\bm{u}^h)$, i.e., $T_F$ is symmetric.
For \textbf{T.2} we use that $\bm{n}_{ij} = -\bm{n}_{ji}$, and so, 
$$
(\bm{\mu}_{ij})_s  
= \int_{\bm{f}_{ij}}\! {\bm{\phi}_s}\cdot \bm{n}_{ij}\, dS 
= -\int_{\bm{f}_{ji}}\! {\bm{\phi}_s}\cdot \bm{n}_{ji}\, dS
=-(\bm{\mu}_{ji})_s;  \quad s=1,\ldots,n_1 \,. 
\vspace{-4ex}
$$
\end{proof}

\begin{lemma}\label{lem:GMLS-Lip}
The pair $\{T_F, \{\bm{\mu}_{ij},\bm{\mu}_{\Gamma_i}\}\}$ is locally Lipschitz continuous on $X$. 
\end{lemma}
\begin{proof}
Let $\bm{u},\bm{v} \in (C^2(\Omega))^d$ with point samples $\bm{u}^h,\bm{v}^h\in U^h_\gamma$. Given   $\bm{f}\in F \cup \{\Gamma_j \,|\, \bm{x}_j\in\gamma^\prime\}$ we have that
\begin{equation}\label{eq:lip1}
\bm{c}_{\bm{f}}(\bm{u}^h)\cdot \bm{\mu}_{\bm{f}}
=
\bm{c}_{\bm{f}}(\bm{u}^h)\cdot \tau_{\bm{f}}(\bm{\phi})
=
\sum_{s=1}^{n_1} \bm{c}^s_{\bm{f}} (\bm{u}^h) \int_{\bm{f}} \bm{\phi}_{s}(\bm{x})\cdot\bm{n} dS .
%= \int_{\bm{f}} \left[ \sum_{s=1}^{n_1} \bm{c}^s_{\bm{f}}(\bm{u}^h) \bm{\phi}_{s}(\bm{x}) \right] \cdot\bm{n} dS
\end{equation}
By assumption, all dual faces are affine and so their unit normals are constant.  As a result, $\bm{\phi}_{s}(\bm{x})\cdot\bm{n}$ is linear on $\bm{f}$ and
$$
\int_{\bm{f}} \bm{\phi}_{s}(\bm{x})\cdot\bm{n} dS =  \bm{\phi}_{s}(\bm{x}_{\bm{f}})\cdot\bm{n} |\bm{f}| \,.
$$
It follows that 
$$
\bm{c}_{\bm{f}}(\bm{u}^h)\cdot \bm{\mu}_{\bm{f}}
=
\left[ \sum_{s=1}^{n_1} \bm{c}^s_{\bm{f}} (\bm{u}^h) 
 \bm{\phi}_{s}(\bm{x}_{\bm{f}})\right] \cdot\bm{n} |\bm{f}|
$$ 
The expression in the square brackets is the Moving Least Squares (MLS) approximant $\widetilde{\bm{u}}(\bm{x}_{\bm{f}})$ of $\bm{u}(\bm{x}_{\bm{f}})$; see \cite[Chapter 4]{WendlandBook}. Using the fact that $\widetilde{\bm{u}}(\bm{x}_{\bm{f}})$ is uniformly bounded by the data we obtain
$$
\left|
\bm{c}_{\bm{f}}(\bm{u}^h)\cdot \bm{\mu}_{\bm{f}} - \bm{c}_{\bm{f}}(\bm{v}^h)\cdot \bm{\mu}_{\bm{f}}
\right|
=
\left| 
\widetilde{\bm{u}}(\bm{x}_{\bm{f}})\cdot\bm{n} - \widetilde{\bm{v}}(\bm{x}_{\bm{f}})\cdot\bm{n}
\right| \,|\bm{f}|
\le
C\|\bm{u}^h - \bm{v}^h \|_{\ell^\infty, X_{B^h_{\bm{f}}}}
|\bm{f}| \,.
$$
The proof follows by noting that for physical faces $|\bm{f}| = O(h^{d-1}_X)$.
\end{proof}

Lemmata \eqref{lem:GMLS-P1}--\eqref{lem:GMLS-Lip} verify all requirements of the abstract theory in \S\ref{sec:theory-abstract-MMD}. Thus, we have the following result.

\begin{theorem}\label{thm:MMD-GMLS}
The operator \eqref{eq:div-dis-GMLS}  satisfies \eqref{eq:mimetic_property} (local conservation), \eqref{eq:mimetic_property_global} (global conservation), and is first-order accurate.
\end{theorem}
%
%%%%%%%%%%%%%%%%%%%%
%\begin{figure}[h!]
%  \centering
%  \includegraphics[height=1.5in]{figures/fvmdiagram.png} 
%  \caption{A sketch of the reconstruction process for a GMLS quadrature rule. Particles in the support of $\omega_i$ (dashed circle) are used to reconstruct $u$ over the $\omega_i$ (barycenter labeled blue).}
%  \label{fvcartoon}
%\end{figure}
%%%%%%%%%%%%%%%%%%%%%

%%%%%%%%%%%%%%%%%%%%
\subsection{MMD without a background mesh}\label{sec:divtheorem}
%%%%%%%%%%%%%%%%%%%%
%
In this section we drop the assumption of a background mesh and only require that the virtual vertices $V$ are tied to the point cloud $X$, see \S\ref{sec:abstract}. 
To implement \eqref{eq:mimetic_div_abs} under these conditions one has to construct the necessary topological, field, and metric data without referencing physical cells and faces, except for the boundary segments $\Gamma_i\in\gamma^\prime$ where field and face moments can be defined as in \S\ref{sec:gmlsquad}.
%

%%%%%%%%
\subsubsection{Topological data}
%%%%%%%%
The virtual boundary in \eqref{eq:virtual_bdry} is completely determined by the list of edges connected to the primal vertex $\bm{v}_i$. Any graph  $G(V,E)$ built on the point cloud can provide this information.  In particular, here we shall use the  $\varepsilon_g$-ball graph $G_{\varepsilon_g} (V,E)$ of  $X$ with vertices $V:={X}$ and edges
$$
E:=\left\{ \bm{e}_{ij} = (\bm{x}_i,\bm{x}_j) \in V\times V \,\big|\,  |\bm{x}_i-\bm{x}_j | < \varepsilon_g\right\} \,,
$$
where $\varepsilon_g$ is a given positive real number. 
This graph  may be trivially constructed with $O(p_{\Omega})$ computational complexity using standard binning algorithms. 
We view $G_{\varepsilon_g} (V,E)$ as a surrogate for a primal mesh and associate it with the virtual primal grid $\tau_h(V,E)$. The latter induces a virtual dual grid $\tau^\prime_h(C,F)$ that does not possess any physical entities, except for the boundary segments $\Gamma_i$.
%which coincide with the edge midpoints  $\bm{x}_{ij} = (\bm{x}_i + \bm{x}_j)/{2}$.

\begin{remark}\label{rem:primal_dual}
We use a (virtual) primal-dual setting to define the abstract MMD operator precisely because any graph on $X$ induces a boundary operator on $\tau^\prime_h(C,F)$.
Building such a graph only requires defining its edges, which can be done efficiently on any point cloud. In contrast, if one were to adopt a single virtual mesh, one would have to construct suitable surrogates for its virtual vertices, edges, faces and cells in order to obtain a correct notion of a virtual boundary. This is a substantially more difficult task. 
\end{remark}

%%%%%%%%%%%
\subsubsection{Field data}\label{sec:field-data}
%%%%%%%%%%%
In Section \ref{sec:gmlsquad} we used the GMLS coefficients  $\bm{c}_{ij}(\bm{u}^h)\in\mathbb{R}^{n_{1}}$ to define the data transfer operator $T_{F} : U^h_\gamma \rightarrow \widetilde{U}^h_\gamma$. 

The only physical mesh entities needed to set up the least-squares problem  \eqref{eq:gmls-P1} for these coefficients were the dual face centroids $\bm{x}_{\bm{f}_{ij}}$, involved in the kernel  $\omega(\tau_{ij},\lambda_i)$.
In the present context the dual faces and their centroids are purely virtual. These virtual centroids can be tied to any reasonable physical location. Here we choose the midpoint  $\bm{x}_{ij} = (\bm{x}_i + \bm{x}_j)/{2}$ of the primal edge corresponding to $\bm{f}_{ij}$ as a physical location of the face's virtual centroid, i.e., we set  $\bm{x}_{\bm{f}_{ij}}:=\bm{x}_{ij} $.

Efficient computation and storage of the field moments requires some care though. The number of virtual faces in $\partial\omega_i$ equals the valence of the primal vertex $\bm{x}_i$. Since in the absence of a background mesh there are no restrictions on the number of edges connected to this vertex, computing and storing the GMLS coefficients $\bm{c}_{ij}(\bm{u}^h)$ for every virtual dual face may become ineffective.  
To improve efficiency we compute and store these coefficients on the point cloud. This can be accomplished by changing the kernel from  $w(\tau_{\bm{f}};\lambda_i) = \kappa(|\bm{x}_{\bm{f}} - \bm{x}_i |)$ to $\omega(\lambda_j,\lambda_i)=\kappa(|\bm{x}_{j} - \bm{x}_i |)$, which replaces the matrix $W(\tau_{\bm{f}})$ in \eqref{eq:gmls-P1} by a new matrix $W(\bm{x}_j)$ with element $W_{ii}(\bm{x}_j) = \kappa(|\bm{x}_{j} - \bm{x}_i |)$. Solution of the modified problem \eqref{eq:gmls-P1} yields the  ``nodal'' GMLS coefficients $\bm{c}_{i}(\bm{u}^h)$, which are then "blended" into the field moments according to 
\begin{equation}\label{thetaeq}
  \bm{c}_{ij}(\bm{u}^h) = \theta_{ij} \bm{c}_{i}(\bm{u}^h) + (1 - \theta_{ij}) \bm{c}_{j}(\bm{u}^h);
  \quad \theta_{ij} = \theta_{ji} \in [0,1].
\end{equation}
This blending preserves polynomial reproduction property and the blended coefficients \eqref{thetaeq}  satisfy \textbf{T.3}.

%%%%%%%%%%%
\subsubsection{Metric data}\label{sec:metric}
%%%%%%%%%%%
In Section \ref{sec:gmlsquad} we used the GMLS target vector $\bm{\tau}_{\bm{f}}(\bm{\phi})\in \mathbb{R}^{n_1}$ to define the virtual face moments $\bm{\mu}_{\bm{f}}$. In contrast to the GMLS coefficient vector $\bm{c}_{ij}(\bm{u}^h)$, which only requires the centroids  $\bm{x}_{\bm{f}_{ij}}:=\bm{x}_{ij} $, computation of $\bm{\tau}_{ij}(\bm{\phi})$ is impossible without physical dual faces. 
Likewise, in the absence of physical dual cells we also need a method for assigning a volume $\mu_i$ to every virtual dual cell. 

In this section we formulate a scalable and computationally efficient algebraic procedure for computing virtual cell volumes  and virtual face moments. 
The procedure starts with construction of virtual volumes $\{\mu_i\}$ satisfying \textbf{T.1}. We then seek the components of the face moments $\bm{\mu}_{ij}$ as gradients of scalar potentials. Inserting this ansatz into the $P_1$-reproduction condition \eqref{eq:mimetic_div_P1} results in $n_{1}$ independent graph Laplacian problems for the face moments that can be solved in an efficient and scalable way by using algebraic multigrid preconditioners.

%%%%%%
\paragraph{Definition of virtual volumes}
%%%%%%
To construct the virtual volumes we consider the quadratic program
\begin{equation}\label{eq:volumes}
\min_{\bm{\mu}\in \mathbb{R}^{p_\Omega}} \sum_{i\in X}\mu_i^2\rho_i 
\quad\mbox{subject to}\quad \sum_{i\in X}\mu_i = |\Omega|\,,
\end{equation}
where $\rho_i>0$ are ``densities'' assigned to each virtual dual cell. In this work we use two different density distributions. 
The first one is  $\rho^{\mbox{\tiny\sf U}}_i = 1$ in which case solution of \eqref{eq:volumes} yields the \emph{uniform} virtual volumes
\begin{equation}\label{eqn:vol1}
  \mu_i^{\mbox{\tiny\sf U}} = \frac{|\Omega|}{p_\Omega},\quad i=1,\ldots, p_\Omega.
\end{equation}
 However, the uniform volumes \eqref{eqn:vol1} may not be the best choice for multiresolution problems that call for non-quasiuniform point clouds. For such point clouds one can use the non-uniform density distribution
$$
\rho_i^{\mbox{\tiny\sf MR}}  = \left( \underset{j\in X}{\sum} \kappa(|\bm{x}_i - \bm{x}_j|) \right)^{-1} \,,
$$
where $\kappa(\cdot)$ is the kernel used in the GMLS problem. In this case solution of \eqref{eq:volumes} results in 
\begin{equation}\label{eqn:vol2}
  \mu_i^{\mbox{\tiny\sf MR}}  = \rho_i^{\mbox{\tiny\sf MR}} \left(\frac{|\Omega|}{\underset{j\in X}{\sum} \rho_j^{\mbox{\tiny\sf MR}}}\right) \,.
\end{equation}
Both  \eqref{eqn:vol1} and \eqref{eqn:vol2} satisfy \textbf{T.1} and numerical results in \S\ref{sec:volume-sens} reveal that either one provides first-order convergence on quasi-uniform point clouds. This suggests that the actual volume definition is not critical for the accuracy of the MMD operator on such clouds.
Nonetheless, unless otherwise noted,  we use \eqref{eqn:vol2} because it is simple to implement, while still providing a sense of adaptivity for multiresolution problems.

\paragraph{Definition of virtual face moments}
%%%%%%
For clarity we discuss construction of the face moments assuming that $\gamma=\Gamma$ and $\gamma^\prime = \emptyset$. 
For the same reason it will be convenient to relabel the polynomial basis functions in \eqref{eq:P1V_basis} and their point samples on the cloud using double indices as
$$
\bm{\phi}_{k,r}(\bm{x}) := \mathbf{e}_k\, \phi_r(\bm{x});\quad
k=1,\ldots,d;\quad r=1,\ldots, d+1\,,
$$
and $\bm{\phi}^h_{k,r} \in U^h_\Gamma$, respectively. We also relabel the components of the field moments and the virtual face moments to match the indexing of the basis functions, i.e., we write $(\bm{c}_{ij})_{k,r}$ and $(\bm{\mu}_{ij})_{k,r}$  for the components of $\bm{c}_{ij}$ and $\bm{\mu}_{ij}$, respectively, corresponding to $\bm{\phi}_{k,r}$.

Consistency of the MMD operator requires $\{ T_F,\{\mu_i,\bm{\mu}_{ij} \}\}$ to be a $P_1$-reproducing pair. A sufficient condition for this property is that \eqref{eq:mimetic_div_P1} holds for all basis functions, i.e., we require that for all $\bm{x}_i\in X_0$
\begin{equation}\label{eq:P1-iff}
\sum_{\bm{f}_{ij}\in\partial\omega_i} \bm{c}_{ij}(\bm{\phi}^h_{k,r}) \cdot \bm{\mu}_{ij} 
= \mu_i\nabla\cdot\bm{\phi}_{k,r} \,\big |_{\bm{x}_i} - \int_{\chi_i(\gamma)}\bm{\phi}_{k,r}\cdot\bm{n} dS;
\quad k=1,\ldots,d;\  r=1,\ldots,d+1 \,.
\end{equation}
Note that 
$\nabla\cdot\bm{\phi}_{k,r} = \partial_{k}\phi_r$ and
$\bm{\phi}_{k,r}\cdot\bm{n} = \bm{n}_k \phi_r$, while the $P_1$-reproduction property of GMLS  implies that
$$
\bm{c}_{ij}(\bm{\phi}^h_{k,r}) \cdot \bm{\mu}_{ij}  
=
(\bm{\mu}_{ij})_{k,r} \,.
$$
Therefore, \eqref{eq:P1-iff} is equivalent to the statement that for all $\bm{x}_i\in X_0$ there holds
\begin{equation}\label{eq:mimetic_div_P1_moment}
\sum_{\bm{f}_{ij}\in\partial\omega_i}  (\bm{\mu}_{ij})_{k,r} 
= \mu_i \partial_{k}\phi_r \,\big |_{\bm{x}_i} - \int_{\chi_i(\gamma)}\bm{n}_k \phi_r \,dS;
\quad \quad k=1,\ldots,d;\  r=1,\ldots,d+1\,.
\end{equation}
The sum over $\bm{f}_{ij}$ in \eqref{eq:mimetic_div_P1_moment} is the graph divergence of a vector containing the $(k,r)$th components of $\bm{\mu}_{ij}$. This prompts us to seek each component as a \emph{scaled} graph gradient of a  virtual area potential  $\bm{\psi}^h_{k,r} \in S^h$, i.e.,
\begin{equation}\label{eq:face-moment-ansatz}
(\bm{\mu}_{ij})_{k,r} := 
\left((\bm{\psi}^{h}_{k,r})_i - (\bm{\psi}^{h}_{k,r})_j\right)\cdot \phi_r(\bm{x}_{ij}) \,,
\end{equation}
where $\bm{x}_{ij}$ is the physical centroid of the virtual face $\bm{f}_{ij}$. Inserting the ansatz \eqref{eq:face-moment-ansatz} into \eqref{eq:mimetic_div_P1_moment} yields the following set of linear equations for the scalar potentials: for $r=1,\ldots,d+1$ solve
\begin{equation}\label{eq:face-moment-graph-lap}
\sum_{\bm{f}_{ij}\in\partial\omega_i} 
\left((\bm{\psi}^{h}_{k,r})_i - (\bm{\psi}^{h}_{k,r})_j\right)\cdot \phi_r(\bm{x}_{ij})
=
\mu_i \partial_{k}\phi_r \,\big |_{\bm{x}_i} - \int_{\chi_i(\gamma)}\bm{n}_k \phi_r \,dS
\quad \forall\bm{x}_i\in X_0
\end{equation}
for $k=1,\ldots,d$.
Equation \eqref{eq:face-moment-graph-lap} is a weighted graph Laplacian problem. Thus, to determine the virtual face moments we need to solve $d+1$ graph Laplacians with different weights, each with $d$ different right-hand-sides. We write these linear systems as
$$
\mathcal{L}_r \bm{\psi}_{k,r} = \bm{b}_{k,r}.
$$
The standard $n \times n$ graph Laplacian is an M-matrix, and therefore, it may be solved with $O(n)$ complexity using standard algebraic multigrid. 
For the weighted graph Laplacian $\mathcal{L}_r$ we must ensure that the weights $\phi_r(\bm{x}_{ij})$ preserve the M-matrix structure. With the assumption that the coordinate system is such that all points in the cloud have positive coordinates,  this can be easily accomplished by, e.g., choosing the scalar basis set as $\{\phi_r\}_{r=1}^{d} = \{ 1, x_1,\ldots, x_d\}$. 

The operator $\mathcal{L}_r$ also has a one-dimensional null space comprising the constant vector $\mathbf{1}$, i.e., $\mathcal{L}_{r} \mathbf{1} = 0$. As a result, the $d$ right-hand sides $\bm{b}_{k,r}$,  for every one of the $d+1$ graph Laplacian equations in \eqref{eq:face-moment-graph-lap} are subject to the compatibility condition $\mathbf{1}^\intercal \bm{b}_{k,r} = 0$ for $k=1,\ldots,d$. 
It is easy to see that the volume assumption in \textbf{T.1} is sufficient for this compatibility condition to hold. Indeed, from the fact that $\partial_{k}\phi_r = const$ and $\sum_i \mu_i = |\Omega|$ it follows that  for all $k=1,\ldots,d$ there holds
$$
\mathbf{1}^\intercal \bm{b}_{k,r}
=
\sum_{\bm{x}_i\in X_0} 
\left(\mu_i \partial_{k}\phi_r \,\big |_{\bm{x}_i} - \int_{\chi_i(\gamma)}\bm{n}_k \phi_r \,dS\right)
= 
|\Omega| \partial_{k}\phi_r - \int_{\partial \Omega} \bm{n}_k \phi_r \,dS 
=
|\Omega| \partial_{k}\phi_r - \int_{\Omega} \partial_{k}\phi_r \,dx 
= 0.
$$

%%%%%%
\subsubsection{MMD without a background mesh: setup and properties} 
%%%%%%
Assuming that $\gamma=\Gamma$ we have the following instance of the abstract MMD operator \eqref{eq:mimetic_div_abs}
\begin{equation}\label{eqn:SBPdiv}
(\DIVGMLSP \bm{u}^h)_i  
= 
\frac{1}{\mu_i} \left[
\sum_{\bm{f}_{ij}\in\partial\omega_i}(\bm{c}_{ij}(\bm{u}^h)) \cdot \bm{\mu}_{ij} +
\int_{\chi_i(\gamma)}\bm{u}\cdot\bm{n} dS
\right] \qquad\forall\omega_i\in C \,,
\end{equation}
where the field moments $\bm{c}_{ij}$ are given by \eqref{thetaeq}, the virtual volumes $\mu_i$ are given by \eqref{eqn:vol1} or \eqref{eqn:vol2}, and the virtual face moments $\bm{\mu}_{ij}$ are defined by solving the graph Laplacian systems \eqref{eq:face-moment-graph-lap} for the scalar potentials $\psi^h_{k,r}$ and then using the ansatz \eqref{eq:face-moment-ansatz}.
Thus, although \eqref{eq:div-dis-GMLS} and \eqref{eqn:SBPdiv}   have the same structure, the metric information in the latter is determined in a purely algebraic way without requiring a physical mesh structure. 
It is, therefore, of some interest to examine the comparative costs of setting up these operators. To make this comparison more equitable we do not count the mesh generation towards the setup cost of \eqref{eq:div-dis-GMLS}, as the purpose of this operator is to improve discretization robustness and accuracy on an already existing mesh. 

The setup of \eqref{eq:div-dis-GMLS} involves computing the GMLS coefficients and the face moments for every dual face. Assuming that the valence of every primal node in the mesh is bounded by a constant $C_{\bm{e}}$ the total number of these faces is\footnote{Recall that the background mesh nodes are assumed to coincide with the point cloud $X$ and so their number equals $p_\Omega$.} $O(p_\Omega)$.
The work to compute each $\bm{c}_{ij}$ is proportional to the average  number $n_\kappa$ of mesh nodes in the support of the kernel function, while the cost to compute every face moment $\bm{\mu}_{ij}$ is proportional to $d+1$. As a result, the setup cost for  \eqref{eq:div-dis-GMLS} is
$O(p_\Omega(n_\kappa+d))$.

The setup cost for \eqref{eqn:SBPdiv} involves computing the nodal GMLS coefficients $\bm{c}_i$ and the scalar potentials $\psi^h_{k,r}$ on the point cloud. The work for every $\bm{c}_{i}$ is again proportional to $n_\kappa$, while the potentials require the solution of $d+1$ linear systems with $d$ different right hand sides per point. Assuming that these systems are solved using algebraic multigrid the cost will scale as $O(p_{\Omega} d^2 )$, where $p_\Omega$ is the number of points in the cloud $X$. Thus, the total setup work for \eqref{eqn:SBPdiv} scales as $O(p_{\Omega}(n_\kappa + d^2 ))$. 

Note that $n_\kappa$ can be assumed fixed as long as the polynomial degree of the GMLS reproducing space $\Phi$ remains the same. Therefore, the setup cost for both \eqref{eq:div-dis-GMLS} and \eqref{eqn:SBPdiv} scales \emph{linearly} with the number of points in the cloud $X$.

\begin{remark}\label{rem:efficient_div}
In the absence of a background mesh the valence of each point in $X$ may grow with the size of the point cloud, leading to an increase of the number of virtual edges in $\partial\omega_i$. 
Using \eqref{thetaeq} to compute the field moments and the scalar potentials $\psi^h_{k,r}$ to compute the virtual face moments allows us to store all field and metric data, required for the evaluation of \eqref{eqn:SBPdiv}, on the point cloud $X$.
 This is essential for ensuring that the setup cost and the storage requirements of \eqref{eqn:SBPdiv}  scale linearly with the size of the point cloud. 
\end{remark}

\begin{remark}\label{rem:AMG}
Although a detailed discussion of a performant implementation of \eqref{eqn:SBPdiv}  is beyond the scope of this work we include in Table \ref{amgscaling} the number of iterations and total wall-clock time  when using preconditioned conjugate gradient method with a classic algebraic multigrid preconditioner  \cite{Bramble_90_MC} to solve  \eqref{eq:face-moment-graph-lap}. 
The information in the table confirms that these standard techniques work well for the weighted graph Laplacian operators in \eqref{eq:face-moment-graph-lap} and achieve linear computational complexity. 
\end{remark}

%
%%%%%%%%%%%%%%%%
\begin{table}[]\label{amgscaling}
\centering
\begin{tabular}{ccc}
$p_\omega$       & Solution time (sec) & AMG iterations \\\hline
$16^2$  & 0.0015           & 5.750           \\
$32^2$  & 0.0038           & 6.125          \\
$64^2$  & 0.0094           & 5.000              \\
$128^2$ & 0.0368           & 6.500            \\
$256^2$ & 0.1839           & 6.625          \\
$512^2$ & 0.5782           & 7.125         
\end{tabular}
\caption{
Average solution time and average number of the AMG iterations for  the weighted graph Laplacian equations \eqref{eq:face-moment-graph-lap} as function of the point cloud size. Solution time and AMG iterations are averaged over the set of $d(d+1)$ solves. Results indicate optimal $O(p_\Omega)$ performance of the AMG solver. 
}
\end{table}
%%%%%%%%%

\paragraph{Properties}
By construction the metric and field data for \eqref{eqn:SBPdiv} satisfies the topological requirements \textbf{T.1}-\textbf{T.3} and $\{T_F\{\mu_i,\bm{\mu}_{ij}\}\}$ is a $P_1$-reproducing pair. This is sufficient for Theorem \ref{thm:mimetic_div_abs_cons} to hold for \eqref{eqn:SBPdiv}, i.e., this MMD instance satisfies \eqref{eq:mimetic_property} (local conservation) and \eqref{eq:mimetic_property_global} (global conservation). On the other hand, to show that \eqref{eqn:SBPdiv} is also first-order accurate requires $\{T_F,\{\bm{\mu}_{ij}\}\}$ to be a locally Lipschitz continuous pair on $X$, i.e., that \eqref{eq:locLip} holds.  The proof of this property is highly non-trivial and will be addressed in a forthcoming paper. 
However, preliminary numerical convergence studies in Section \ref{sec:num} suggest that  \eqref{eqn:SBPdiv} is indeed first-order accurate.

%%%%%%%%%%%%%%%%%%%%%%%%%%%
\section{A virtual finite volume scheme for a scalar advection-diffusion equation}\label{sec:apply}
%%%%%%%%%%%%%%%%%%%%%%%%%%%
In this section we apply the abstract MMD operator \eqref{eq:mimetic_div_abs} to construct a meshfree mimetic discretization of the model boundary value problem
\begin{equation}\label{eq:ad}
\begin{aligned}
  -\nabla \cdot \bm{\sigma}(u) = f &\quad \mbox{in $\Omega$} \\ 
                                            u = g &\quad \mbox{on $\gamma^\prime$} \\
    \bm{n}\cdot \bm{\sigma}(u) = h &\quad \mbox{on $\gamma$}
  \end{aligned}
\end{equation}
where $u$ is a scalar variable, $f$, $g$ and $h$ are prescribed data, and $\sigma(\cdot)$ is a flux function. To illustrate the application of the MMD operator it suffices to consider a simple advective-diffusive flux given by
\begin{equation}\label{adflux}
  \bm{\sigma}(u) = \varepsilon \nabla u - \mathbf{a} u := \bm{\sigma}_d + \bm{\sigma}_a\,,
\end{equation}
where $\varepsilon>0$ is a picewise $C^0$ diffusion coefficient and $\mathbf{a} \in (C^0(\Omega))^d$ is a given velocity field. 

We highlight the ability of our approach to deliver a truly conservative meshfree scheme without a physical background mesh by implementing the second instance \eqref{eqn:SBPdiv} of the abstract MMD operator. Thus, we consider the discretization infrastructure of Section \ref{sec:divtheorem}, comprising a quasi-uniform point cloud $X$, an $\epsilon_g$-ball graph $G_{\epsilon_g}(V,E)$ serving as a surrogate primal mesh $\tau_h(V,E)$, and the associated virtual dual mesh $\tau^\prime_h(C,F)$. 
The virtual volumes $\mu_i$ are set by \eqref{eqn:vol2}, while the virtual face moments $\bm{\mu}_{ij}$ are defined by \eqref{eq:face-moment-ansatz} with potentials determined by solving the graph Laplacian systems \eqref{eq:face-moment-graph-lap}. The face moments on the Dirichlet boundary segments  are computed using the target functionals in 
\eqref{eq:GMLSface}, i.e.,
$$
(\bm{\mu}_{\Gamma_i})_s =\int_{\Gamma_j}\bm{\phi}_s\cdot\bm{n}\, dS;
\quad \mbox{for $s=1,\ldots, n_{1}$ and  $\Gamma_j\subset \gamma^\prime$ } \,,
$$
where $\{\bm{\phi}_s\}$ is a basis of $\Phi = (P_1)^d$. To discretize \eqref{eq:ad} we represent the unknown field $u(\bm{x})$ by its point samples on $X_0\cup X_{\gamma}$, the source term $f$ by its point samples on $X$ and assume that the boundary data  $g$ and $h$ are given on the Dirichlet and Neumann parts of the boundary, respectively. 
We thus obtain the following meshfree mimetic discretization of \eqref{eq:ad}: given $f^h\in S^h$ find $u^h\in S^h_{\gamma^\prime}$ such that
\begin{equation}\label{eq:ad-MMD}
\mu_i (\DIVGMLSP\, \widetilde{\bm{\sigma}}(u^h))_i = \mu_i f_i\quad\forall \bm{x}_i\in X_0\cup X_{\gamma}\,,
\end{equation}
where $\widetilde{\bm{\sigma}}(u^h)\in \widetilde{U}^h_\gamma$ is a  ``numerical flux''. Note that the Dirichlet boundary condition is enforced in \eqref{eq:ad-MMD} through the definition of the space $S^h_{\gamma^\prime}$, which implies that $u^h_i = g(\bm{x}_i)$ for all $\bm{x}_i\in \gamma^\prime$. In contrast, the Neumann boundary condition is enforced through the definition of the MMD operator \eqref{eqn:SBPdiv} by applying the specified boundary data on all segments $\Gamma_j\subset \gamma$. 
To complete the formulation of  \eqref{eq:ad-MMD} it remains to construct a numerical flux $\widetilde{\bm{\sigma}}^h$ from the field data $u^h\in S^h_\gamma$. Therefore, our application of the MMD operator to \eqref{eq:ad} is similar to the second use case described in Remark \ref{rem:field_data}, and that is typical of finite volume schemes. Thus, we shall refer to \eqref{eq:ad-MMD} as a virtual finite volume discretization of the model problem. We discuss construction of the numerical flux in the next section.

%\begin{remark}\label{rem:extension}
%Comment on Extension to a time dependent problem is straightforward and the conservation property in the time dependent case
%
%A typical mimetic-in-space discretization of \eqref{unsteadyCLaw}, based on a primal-dual mesh and the mimetic divergence operator \eqref{eq:mimetic_div}, is given by 
%%
%\begin{equation}\label{eq:mesh-mimetic}
%  \partial_t (\mu_i {\rho}^h_i )
%  = -\mu_i (DIV\, F^h(\rho^h))_i + \mu_i f^h_i
% \end{equation}
% %
% where $\rho^h(t)$ and  $F^h(\rho^h)$ are a cell-centered approximation of $\rho$ and a numerical flux, respectively, which, for every $t\in[0,T]$,  belong in $\widetilde{S}^h$, and $\mathcal{U}^h$, respectively, while $f^h_i(t) \in \widetilde{S}^h$ is a cell centered approximation of the source
% %
% Assuming a vanishing source, Lemma \ref{lem:mimetic_property} implies that the semi-discrete in space solution satisfies  the following global discrete  conservation statement:
%%
%\begin{equation}\label{eq:global-mesh-mimetic}
%  \frac{d}{dt} \sum_{\omega_i\in C} \mu_i {\rho}^h_i 
%  = \sum_{\Gamma_i \in \partial \Omega} \mu_{\Gamma_i} {F}^h(\rho^h) 
%\end{equation}
%%
%In particular, if the model problem \eqref{unsteadyCLaw} is augmented with homogeneous Neumann conditions, \eqref{eq:global-mesh-mimeic} implies that the total discrete mass, i.e., $\sum_{\omega_i\in C} \mu_i {\rho}^h_i$, is preserved in the domain $\Omega$.

%%%%%%%%%%%%%
\subsection{Numerical flux}\label{sec:flux}
%%%%%%%%%%%%%
%
We will define the numerical flux as a sum
$$
\widetilde{\bm{\sigma}}(u^h) := \widetilde{\bm{\sigma}}_d(u^h)+\widetilde{\bm{\sigma}}_a(u^h)\,,
$$
where $\widetilde{\bm{\sigma}}_d(u^h)\in \widetilde{U}^h_\gamma$ and $\widetilde{\bm{\sigma}}_a(u^h)\in \widetilde{U}^h_\gamma$ are diffusive and advective flux moments, respectively. To compute these moments we follow the procedure outlined in Section \ref{sec:field-data} and first generate the necessary GMLS coefficients on the point cloud. Then we use  \eqref{thetaeq} to define $\widetilde{\bm{\sigma}}_d(u^h)$ and $\widetilde{\bm{\sigma}}_a(u^h)$ on the virtual faces. In both cases we assume that the diffusion coefficient and the velocity field are represented on the point cloud by their point samples $\varepsilon^h\in S^h$ and $\mathbf{a}^h\in U^h_\gamma$, respectively.

%%%%%%%%
\subsubsection{Advective flux moments}\label{sec:advective}
%%%%%%%%
The first step in the construction of $\widetilde{\bm{\sigma}}_a(u^h)$ involves computing a set of ``nodal'' GMLS coefficients at all points $\bm{x}_i\in X_0\cup X_\gamma$ from a point sample $\bm{\sigma}^h_a(u)\in U^h_\gamma$ of the advective flux. We consider two different types of such  ``nodal'' coefficients. 
The first one, denoted by $\bm{c}^{a}_{i}(u^h)$, solves the weighted least-squares problem \eqref{eq:gmls-P1} with weight matrix $W(\bm{x}_i)$ and a point sample\footnote{Here $\circ$ denotes the Hadamar (pointwise) product of two vectors.} $\bm{\sigma}^h_a(u)= \mathbf{a}^h\circ u^h \in U^h_\gamma$, i.e., 
\begin{equation}\label{eq:gmls-aflux}
\bm{c}_{i}^a({u}^h)
= \argmin_{\bm{b}\in\mathbb{R}^{n_{1}}}\frac12
\left | B\bm{b} - \bm{\sigma}_a^h(u)  \right |^2_{W(\bm{x}_i)} \,.
\end{equation}
The second ``nodal'' GMLS coefficient, denoted by $\maxvec{\bm{c}}^{\,a}_{i}(u^h)$, is an upwind version of $\bm{c}^{a}_{i}(u^h)$ intended for advection-dominated problems. To incorporate upwinding we modify \eqref{eq:gmls-aflux} as follows. 
Given a point $\bm{x}_i\in X_0\cup X_\gamma$ we define the \emph{upwind part} of the point cloud at this point as
$$
\maxvec{X}_i  = \{\bm{x}_k\in X\,|\, \mathbf{a}_i\cdot\left(\bm{x}_k-\bm{x}_i\right) < 0 \} \cup\{\bm{x}_i\}\,.
$$
Thus, $\maxvec{X}_i$ is the subset of $X$ containing the point $\bm{x}_i$ and all cloud points \emph{strictly upwind} of that point. 
Next, we define the \emph{upwind point sample} of the advective flux as 
$$
\maxvec{\bm{\sigma}}^h_a(u) = \{ (\bm{\sigma}^h_a(u))_i \,|\, \bm{x}_i \in \maxvec{X}_i \} \,.
$$
Finally, we redefine the basis sample matrix $B$ and the weight matrix $W(\bm{x}_i)$ in \eqref{eq:gmls-aflux} to include only the rows corresponding to the points in $\maxvec{X}_i$ and denote the new matrices as $\maxvec{B}_i$ and $\maxvec{W}(\bm{x}_i)$, respectively. The upwind ``nodal'' GMLS coefficient is then obtained by solving the modified weighted least-squares problem
$$
\vec{\bm{c}}_{i}^{\,a}({u}^h)
= \argmin_{\bm{b}\in\mathbb{R}^{n_{1}}}\frac12
\left | \maxvec{B}_i\bm{b} - \maxvec{\bm{\sigma}}^h_a(u)  \right |^2_{\maxvec{W}(\bm{x}_i)} \,.
$$

We use \eqref{thetaeq} with $\theta_{ij}=1/2$ and the first type of ``nodal'' GMLS coefficients to define the \emph{standard} advective flux moments $\widetilde{\bm{\sigma}}_a ({u}^h)$. Setting $\theta_{ij}=1$ if $\mathbf{a}(\bm{x}_{ij})\cdot\left(\bm{x}_j-\bm{x}_i\right) \geq 0$ and $\theta_{ij}=0$ otherwise, and using the second type of GMLS coefficients defines the \emph{upwind} moments $\maxvec{\widetilde{\bm{\sigma}}}_{a}({u}^h)$. 
Succinctly,  
\begin{equation}\label{eq:upwind}
\left(\widetilde{\bm{\sigma}}_a({u}^h) \right)_{ij} = \frac12\left(\bm{c}_{i}^a({u}^h) + \bm{c}_{j}^a({u}^h)\right)
\quad
\mbox{and}\quad
\left(\maxvec{\widetilde{\bm{\sigma}}}_{a}({u}^h)\right)_{ij} =
\begin{cases}
    \vec{\bm{c}}_{i}^{\,a}({u}^h)     & \mbox{if $\mathbf{a}(\bm{x}_{ij})\cdot\left(\bm{x}_j-\bm{x}_i\right) \geq 0$}\\[1ex]
   \vec{\bm{c}}_{j}^{\,a}({u}^h)      & \mbox{if $\mathbf{a}(\bm{x}_{ij})\cdot\left(\bm{x}_j-\bm{x}_i\right) < 0$} 
\end{cases} \,.
\end{equation}

%%%%%%%%
\subsubsection{Diffusive flux moments}\label{sec:difflux}
 %%%%%%%%
As in the advective flux case we obtain  $\widetilde{\bm{\sigma}}_d(u^h)$ by blending  ``nodal'' GMLS coefficients. 
However, generating these coefficients requires a different approach because the point values of $\nabla u$ are not specified on the cloud and so, a point sample of $\bm{\sigma}_d(u)$ is not readily available on $X$. This means that we cannot use the setting of  \eqref{eq:gmls-aflux}, which requires such a sample.  
Instead, we shall obtain the ``nodal'' coefficients through an auxiliary GMLS problem for the scalar field $u$ that is exact for quadratic polynomials. Specifically, we consider the GMLS framework in Section \ref{sec:GMLS} with $U=C^1(\Omega)$, $\Phi = P_2$, sampling functionals $\lambda_i=\delta_{\bm{x}_i}$, a target $\tau_{i}(u) =\bm{\sigma}_d(u)|_{\bm{x}_i} =\varepsilon\nabla u(\bm{x}_i)$, and kernel $w(\tau_i,\lambda_j) = \kappa(|\bm{x}_i-\bm{x}_j|)$. 
The GMLS approximation of $\tau_i(u)$ at  $\bm{x}_i\in X$ is then given by
\begin{equation}\label{eq:gmls-grad}
(\bm{\sigma}_d(u^h))_{i} = \bm{c}({u}^h)\cdot (\varepsilon \nabla \phi(\bm{x}_i))
\end{equation}
where $\phi = \{\phi_q\}_{i=1}^{n_2}$ is basis of $P_2$, 
$\nabla \phi(\bm{x}_i) = (\nabla \phi_1(\bm{x}_i),\ldots,\nabla \phi_{n_2}(\bm{x}_i))$, and 
\begin{equation}\label{eq:gmls-dflux1}
\bm{c}_{i}({u}^h)
= \argmin_{\bm{b}\in\mathbb{R}^{n_2}}\frac12
\left | B\bm{b} - u^h  \right |^2_{W(\bm{x}_i)} \,.
\end{equation}
Since $\nabla{\phi}_q \in (P_1)^d$ for all $q=1,\ldots,n_2$, it follows that there exists a matrix $G$ such that
$$
\nabla {\phi} = G \bm{\phi}
$$
where $ \bm{\phi} = \{\bm{\phi}_k\}_{k=1}^{n_1}$ is basis of $(P_1)^d$.  As a result, we can express the target approximation \eqref{eq:gmls-grad} in terms of this basis as 
\begin{equation}\label{eq:gmls-grad-p1}
(\bm{\sigma}_d(u^h))_{i} = \varepsilon_i \bm{g}_{i}({u}^h)\cdot \bm{\phi}(\bm{x}_i)\,,
\end{equation}
where $\bm{g}_{i}({u}^h) = G^\intercal \bm{c}_{i}({u}^h)$ and $\varepsilon_i=\varepsilon |_{\bm{x}_i}$. We refer to $\bm{g}_{i}({u}^h)$ as the \emph{derived} ``nodal'' GMLS coefficients because they are not based on a point sample of the diffusive flux but rather derived from a point sample of the scalar field $u$.  To define the diffusive flux moments on the virtual faces we then use \eqref{thetaeq} with  $\theta_{ij}=1/2$, the derived ``nodal'' coefficients, and an averaged diffusion coefficient $\varepsilon_{ij}$:
 $$
 \left(\widetilde{\bm{\sigma}}_d(u^h)\right)_{ij}
 =
 \frac{\varepsilon_{ij}}{2}\left(\bm{g}_{i}({u}^h) + \bm{g}_{j}({u}^h) \right) \,.
 $$
 In  Section \ref{sec:darcy-hdiv}, we will explore the use of both the arithmetic mean $\varepsilon^{A}_{ij} = (\varepsilon_i+\varepsilon_j)/2$ and the harmonic mean $\varepsilon^{H}_{ij} = 2\varepsilon_i \varepsilon_j/(\varepsilon_i+\varepsilon_j)$, which is more appropriate when $\varepsilon$ has large jumps.

%%%%%%%%%%%%%%%%%%%%%
\section{Numerical results}\label{sec:num}
%%%%%%%%%%%%%%%%%%%%%

In the following sections, we study numerically the properties of the virtual finite volume scheme \eqref{eq:ad-MMD} for the model problem \eqref{eq:ad}. We consider two versions of this problem corresponding to $\mathbf{a}=\mathbf{0}$ and $\mathbf{a}\neq \mathbf{0}$, respectively. We refer to the first version as the \emph{Darcy} problem and to the second one as the \emph{advection-diffusion} problem.  For simplicity, we restrict attention to spatial domains in $\mathbb{R}^2$. 

Unless otherwise noted, we generate the point clouds for our study using the following procedure. Given a desired fill distance $h_X>0$ we first construct a uniform partition of the boundary comprising segments $\Gamma_i$ of length $h_X$. The barycenters of these segments define ${X}_{\Gamma}$. Then, we  obtain ${X}_0$ by generating a uniform lattice $\mathcal{L}_{h_X}$ of spacing $O(h_X)$ over a square containing $\Omega$, and finally define ${X}_0 = \mathcal{L}_{h_X} \cap \Omega$.

 To remove any symmetries that might prompt superconvergence, we then perturb each point in ${X}_0$ by a uniformly distributed random variable with magnitude $h_X/5$. In so doing we obtain point clouds satisfying the assumption of quasi-uniformity \eqref{eq:quasi}. Figure \ref{fig:diskgeom} shows a point cloud on the unit disk generated by this procedure. To facilitate post-processing of results, it will often be convenient to generate the Delaunay triangulation corresponding to $X$; it will be made explicit if results are presented on the Delaunay triangulation rather than natively on the point cloud. 
%%%%%%%%%%%%%%%%
\begin{figure}[h!]
  \centering
  \includegraphics[width=0.7\textwidth]{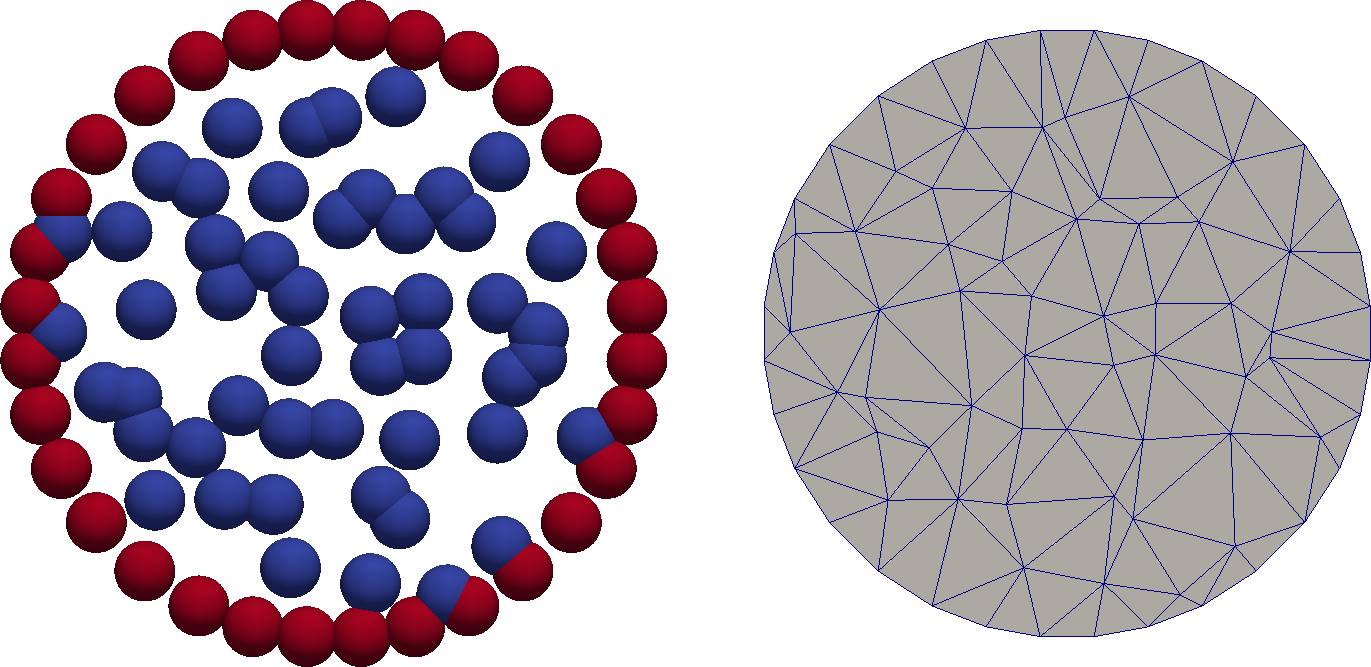}
  \caption{\textit{Left:} Particle distribution for $h_X = 1/8$. The red and blue particles represent $X_\Gamma$ and $X_0$, respectively.  \textit{Right:} a Delaunay triangulation corresponding to $X$ that may be used to display results. }
  \label{fig:diskgeom}
\end{figure}
%%%%%%%%%%%%%%%%

%%%%%%
\subsection{Darcy problem: convergence study}\label{sec:convergence}
%%%%%%
%%%%%%%%%%%%%%%%
\begin{figure}[h!]
  \centering
  \includegraphics[width=0.7\textwidth]{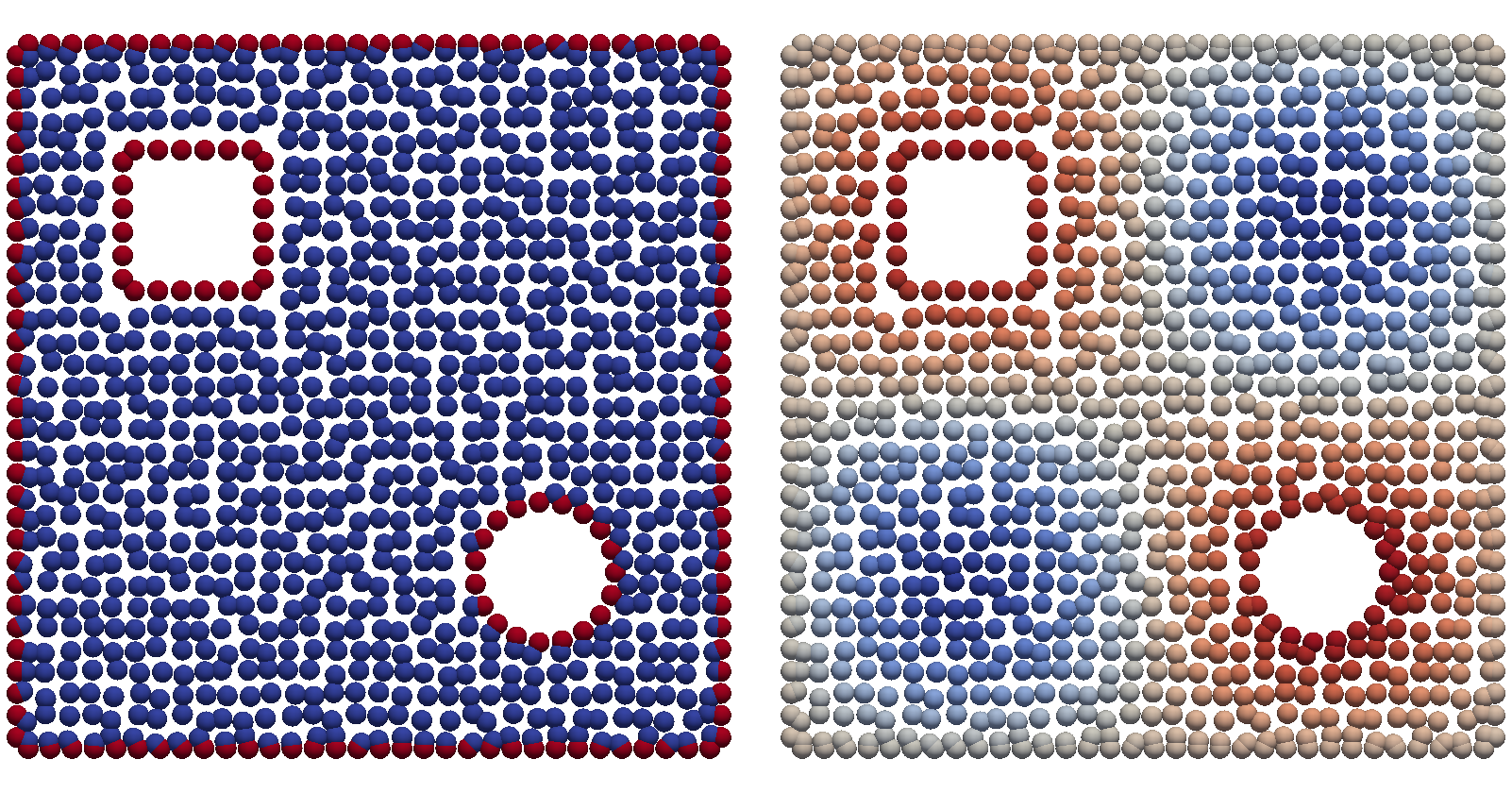}
  \caption{
\textit{Left:} point cloud $X$ generated using a uniform $32\times 32$ lattice. The red and blue particles represent $X_\Gamma$ and $X_0$, respectively.  
\textit{Right:} Approximation to $u_{ex} =  \sin 2 \pi x \sin 2 \pi y$ computed by the virtual finite volume scheme \eqref{eq:ad-MMD}.}
\label{fig:holey}
\end{figure}
%%%%%%%%%%%%%%%%
We perform a numerical convergence study of  \eqref{eq:ad-MMD} using \eqref{eq:ad} with $\varepsilon=1$, $\mathbf{a}=\mathbf{0}$, pure Neumann and pure Dirichlet boundary conditions,  the smooth manufactured solution $u_{ex}(x,y) = \sin 2 \pi x \sin 2 \pi y$, and the domain shown the  in Figure \ref{fig:holey}. Substitution of $u_{ex}(x,y)$ into the governing equations defines the necessary data fields for the study. 
In the pure Neumann case the solution is determined up to an arbitrary constant. To ensure uniqueness
we use  a Lagrange multiplier to enforce a zero mean condition on the solution. We refer to  \cite{Bochev_05_SIREV} and the references therein for further information on handling pure Neumann conditions. 

To estimate convergence rates we use successively refined point clouds constructed from a sequence of $N\times N$ uniform lattices for $N=16,32,64,128,256$. Rates are estimated with respect to the error measures
$$
\| u^h - u^h_{ex} \|^2_{\ell^2,X}:= \sum_{i\in X} (u^h_i - (u^h_{ex})_i)^2 
\quad\mbox{and}\quad
|  u^h - u^h_{ex} |^2_{1,X}:=\frac{1}{\varepsilon} \sum_{i\in X} |(\bm{\sigma}_d(u^h))_i - (\bm{\sigma}_d(u_{ex}))_i|^2 \,,
$$
respectively, where $u^h_{ex}$ is a point sample of the manufactured solution and  $\bm{\sigma}_d(u^h)$ is the GMLS reconstruction of the diffusive flux defined in \eqref{eq:gmls-grad}. We refer to these measures as  the $\ell_2$-norm and the $h_1$-seminorm, respectively. Table \ref{darcySmooth} shows convergence results for the virtual finite volume solution with pure Dirichlet and pure Neumann boundary conditions in these norms.
We observe roughly the same convergence rate in both norms, and for both types of boundary conditions, which confirms numerically that \eqref{eq:ad-MMD} is first-order accurate. 

%
%%%%%%%%%%%%%%%%%
\begin{table}[]\label{darcySmooth}
\centering
\begin{tabular}{ccccc}
{N}            & \textbf{$\ell_2$-Dirichlet} & \textbf{$h_1$-Dirichlet} & \textbf{$\ell_2$-Neumann} & \textbf{$h_1$-Neumann} \\ \hline
$16$  & 0.0983022             & 0.193076              & 0.494616              & 0.353165   \\
$32$  & 0.0516438             & 0.0943233             & 0.307722              & 0.200986   \\
$64$  & 0.0284723             & 0.0484218             & 0.165076              & 0.103267   \\
$128$ & 0.0184388             & 0.0288341             & 0.0771193             & 0.0512336  \\
$256$ & 0.00831139            & 0.0132393             & 0.0405609             & 0.0257531  \\ \hline
Rate            & 1.14958               & 1.12295               & 0.927002              & 0.992344  
\end{tabular}
\caption{Convergence rates of the virtual finite volume scheme \eqref{eq:ad-MMD} for the manufactured solution $u_{ex}(x,y) = \sin 2 \pi x \sin 2 \pi y$.}
\end{table}
%%%%%%%%%%%%%%%%%

%%%%%%
\subsection{Darcy problem: $H(div)$-compatibility study}\label{sec:darcy-hdiv}
%%%%%%
In this section we consider Darcy problems with discontinuous coefficients $\varepsilon$.
In this case, for any interface $\eta$ corresponding to a discontinuity in $\varepsilon$, solutions of \eqref{eq:ad} must satisfy the transmission conditions
\begin{equation}\label{jumpCond}
u_{+} = u_{-}\quad\mbox{and}\quad
\bm{n}_{\eta} \cdot \varepsilon_+ \nabla u_{+}= \bm{n}_{\eta} \cdot  \varepsilon_{-}  \nabla u_{-},
\quad\mbox{on $\eta$},
\end{equation}
where $\bm{n}_{\eta}$ is a unit normal to $\eta$ and $+$ and $-$ denote entities associated with the opposite sides of the interface.
Conditions \eqref{jumpCond} imply that  the normal component of the flux $\varepsilon\nabla u$ is always continuous across the interface but its tangential component and the gradient $\nabla u$ may be discontinuous on $\eta$. 
A hallmark of any $H(div)$-compatible discretization for such Darcy problems is the ability to accurately represent this flux behavior, which is essential for, e.g., accurate simulations of subsurface flow problems.

To evaluate the extent to which the virtual finite volume scheme \eqref{eq:ad-MMD} may be deemed $H(div)$-compatible, we consider three benchmark examples adapted from \cite{Hughes-darcy-dg,masud} and driven by pure Neumann conditions. In each one of these examples $\varepsilon$ is a piecewise constant field defined with respect to a disjoint partition $\Omega = {\cup}_i \Omega_i$ of the unit square. The value of $\varepsilon$ on $\Omega_i$ is denoted by $\varepsilon_i$. Figure \ref{fig:darcyGeom} shows the partitions for each one of the three examples. We refer to these examples as the two strip, five strip and five spot problems, respectively. 
%
% Specifically, we consider a case where the interfaces are orthogonal to the gradient, a case where the interfaces are parallel to the gradient, and a case which is a combination of these examples and whose solution admits a point singularity. 
%
%%%%%%%%%%%%%%%%%%%%
  \begin{figure}[h!]
  \centering
  \includegraphics[width=0.3\textwidth]{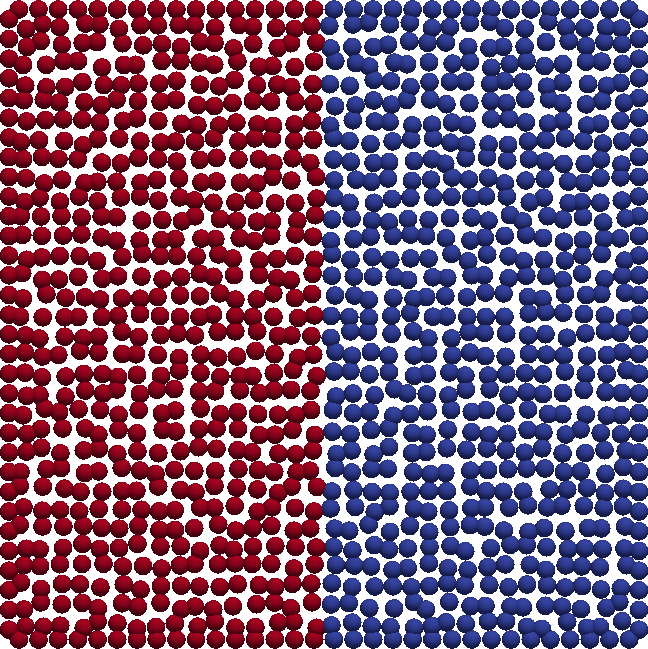}
  \includegraphics[width=0.3\textwidth]{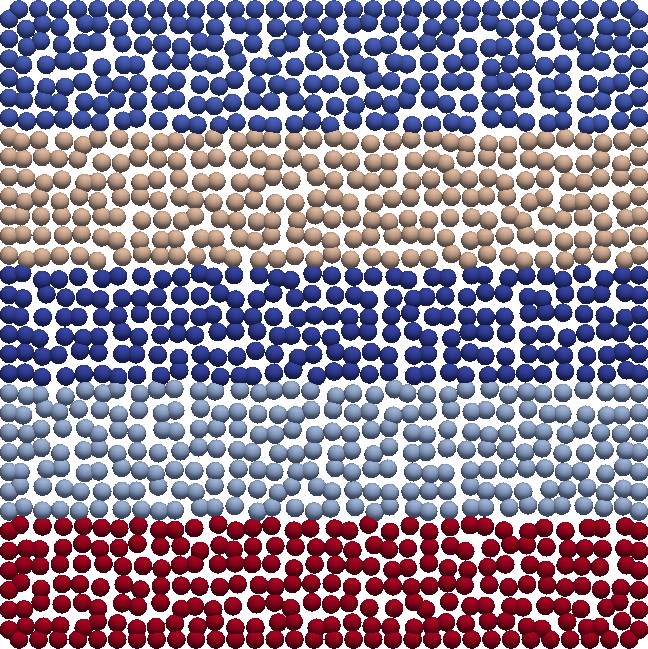}
  \includegraphics[width=0.3\textwidth]{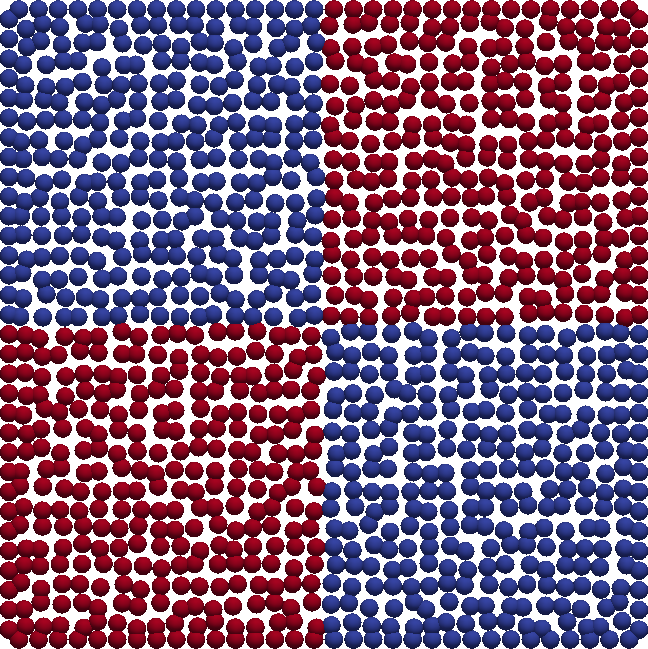}
  \caption{Discontinuous diffusivity fields for Darcy benchmarks. 
   \textit{Left:} The two strip problem. \textit{Center:} The five-strip problem    \textit{Right:} The five-spot problem.}
   \label{fig:darcyGeom}
  \end{figure}
%%%%%%%%%%%%%%%%%%%%

\paragraph{Two strip problem}
In this example $\varepsilon$ is discontinuous along the vertical line $x=1/2$ and 
$$
  u_{ex}(x,y) =
  \begin{cases}
    x & \text{ if } x \leq \frac12 \\
    \frac12 + \frac{\varepsilon_1}{\varepsilon_2} \left( x - \frac12 \right)& \text{ if } x > \frac12 
  \end{cases}
$$
The exact solution satisfies the interface condition \eqref{jumpCond} and the pure Neumann condition
$$
-\varepsilon \nabla u_{ex} \cdot {\bm{n}} = \varepsilon_1\mathbf{e}_1\cdot {\bm{n}}
\quad\mbox{on $\Gamma$}.
$$
where $\bm{n}$ is the outer unit normal to the boundary. Note that the gradient of the exact solution has a discontinuity proportional to the diffusivity ratio $R=\varepsilon_1/\varepsilon_2$. We use this fact to compare and contrast the arithmetic $\varepsilon^{A}_{ij}$ and geometric $\varepsilon^{M}_{ij}$ mean reconstructions of the diffusivity defined in \S\ref{sec:difflux}.
Figure \ref{fig:normal1} compares approximations of $\partial_x u_{ex}$ computed by the virtual finite volume scheme using these two reconstructions. The solution obtained with $\varepsilon^{A}_{ij}$ exhibits oscillations in the vicinity of the discontinuity, while using the harmonic mean allows the scheme to accurately represent the discontinuity in the solution. Figure \ref{fig:normal2} compares solution profiles along $x=0.5$ for increasing diffusivity ratios $R$. We see that the results are comparable for $\varepsilon_1/\varepsilon_2 = O(1)$, but as the ratio is increased the harmonic mean is better able to preserve the location of the jump. Motivated by these results, we use the harmonic mean for all results in the remainder of this paper.
%%%%%%%%%%%%%%%%%%%
  \begin{figure}[h!]
  \centering
   \includegraphics[width=0.45\textwidth]{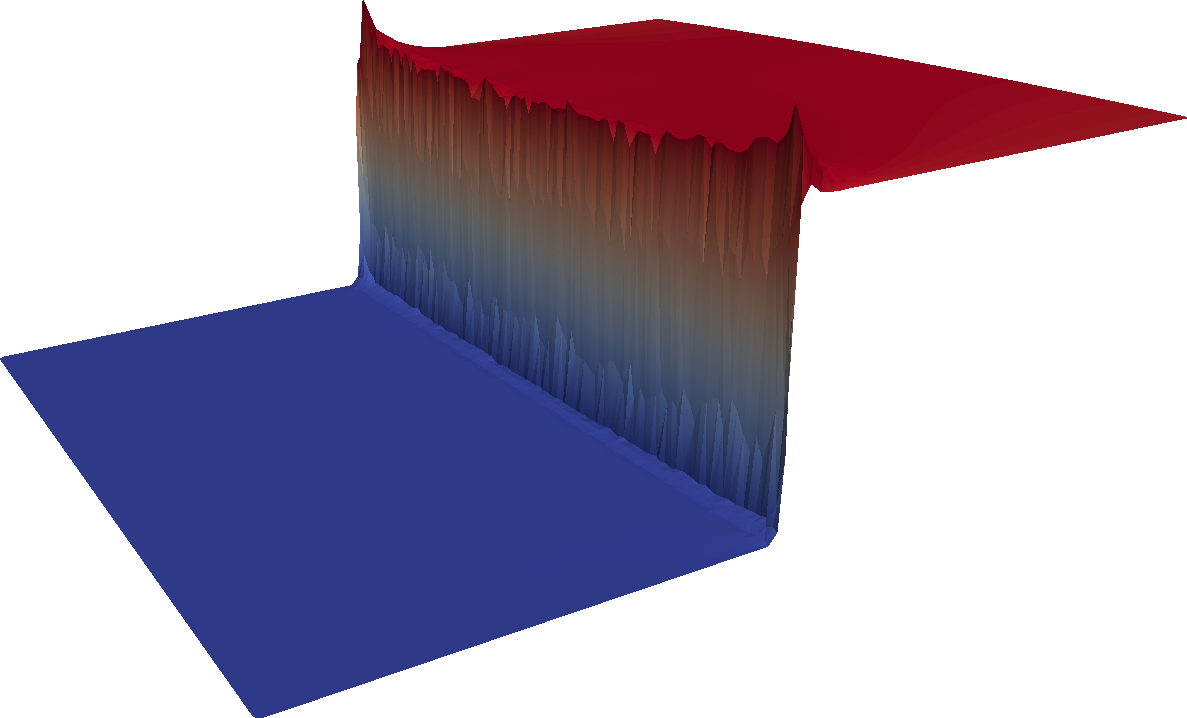}
   \includegraphics[width=0.45\textwidth]{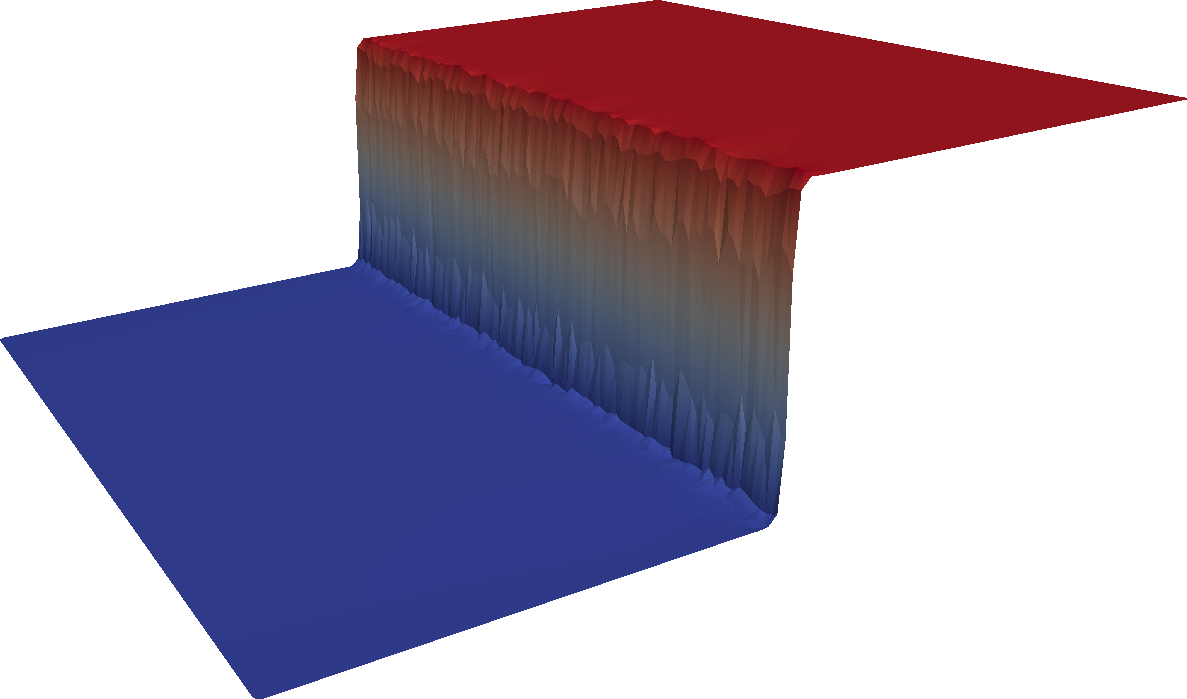}
   \caption{Approximation of $\partial_x u_{ex}$ by the virtual finite volume scheme \eqref{eq:ad-MMD} for $92 \times 92$ particles and diffusivity ratio $R = 64$. \textit{Left}  arithmetic average.  \textit{Right}: harmonic average.}
    \label{fig:normal1}
\end{figure}

  \begin{figure}[h!]
  \centering
   \includegraphics[width=0.45\textwidth]{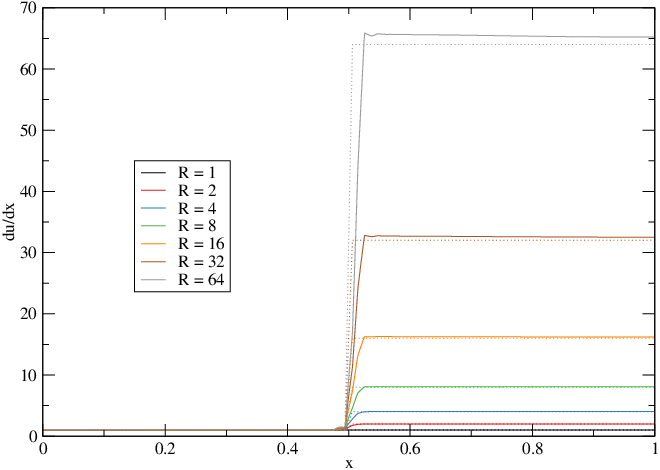}
   \includegraphics[width=0.45\textwidth]{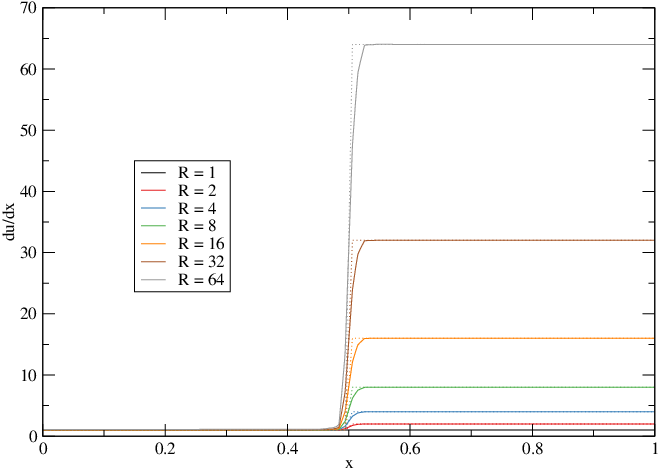}
   \caption{Profile of the virtual finite volume approximation of $\partial_x u_{ex}$  along the line $x = 0.5$ for increasing diffusivity ratios. \textit{Left:}  arithmetic average.  \textit{Right}: harmonic average. Exact solution is given by dashed line.}
   \label{fig:normal2}
\end{figure}
%%%%%%%%%%%%%%%%%%%

\paragraph{The five-strip problem} 
In this example $\varepsilon$ has discontinuities along the lines $y=0.2$, $y=0.4$, $y=0.6$ and $y=0.8$, which divide the unit square into five horizontal strips. We set the diffusivity value on each strip  to
$\varepsilon_1=16$, $\varepsilon_2=6$, $\varepsilon_2=1$, $\varepsilon_4=10$, and $\varepsilon_5=2$, respectively and let
$$
u_{ex} = 1-x.
$$
This exact solution satisfies the interface conditions \eqref{jumpCond}, the pure Neumann condition
$$
\nabla u_{ex} \cdot {\bm{n}} = \mathbf{e}_1 \cdot {\bm{n}},
$$
has a continuous vertical flux component $\varepsilon\partial_{y}u_{ex}=0$, and a piecewise constant horizontal flux component $\varepsilon\partial_{x}u_{ex}=\varepsilon_i$ on $\Omega_i$. The five strip benchmark is designed to test how well a scheme can represent this flux behavior. Figure \ref{fig:5strip1} shows the approximation of $\varepsilon\partial_{x}u_{ex}$ for both uniform and non-uniform point clouds. The profiles of the horizontal flux approximation along $x=0.5$ for increasing point cloud resolutions are shown in Figures \ref{fig:5strip2}. These results show that the virtual finite volume scheme \eqref{eq:ad-MMD} can accurately represent the flux components on both uniform and non-uniform point clouds. 
In particular, approximations of the discontinuous, horizontal flux component on uniform point clouds are 
indistinguishable from those obtained by an $H(div)$-conforming, mesh-based scheme using interface-fitted grids. The small wiggles in the non-uniform case are caused by the non-symmetric distribution of points across the interfaces, which is the meshless analogue of an unfitted grid. We note that an $H(div)$ scheme on an unfitted grid would have a similar behavior. 

%%%%%%%%%%%%%%%%%%%%
\begin{figure}[h!]
  \centering
  \includegraphics[width=0.45\textwidth]{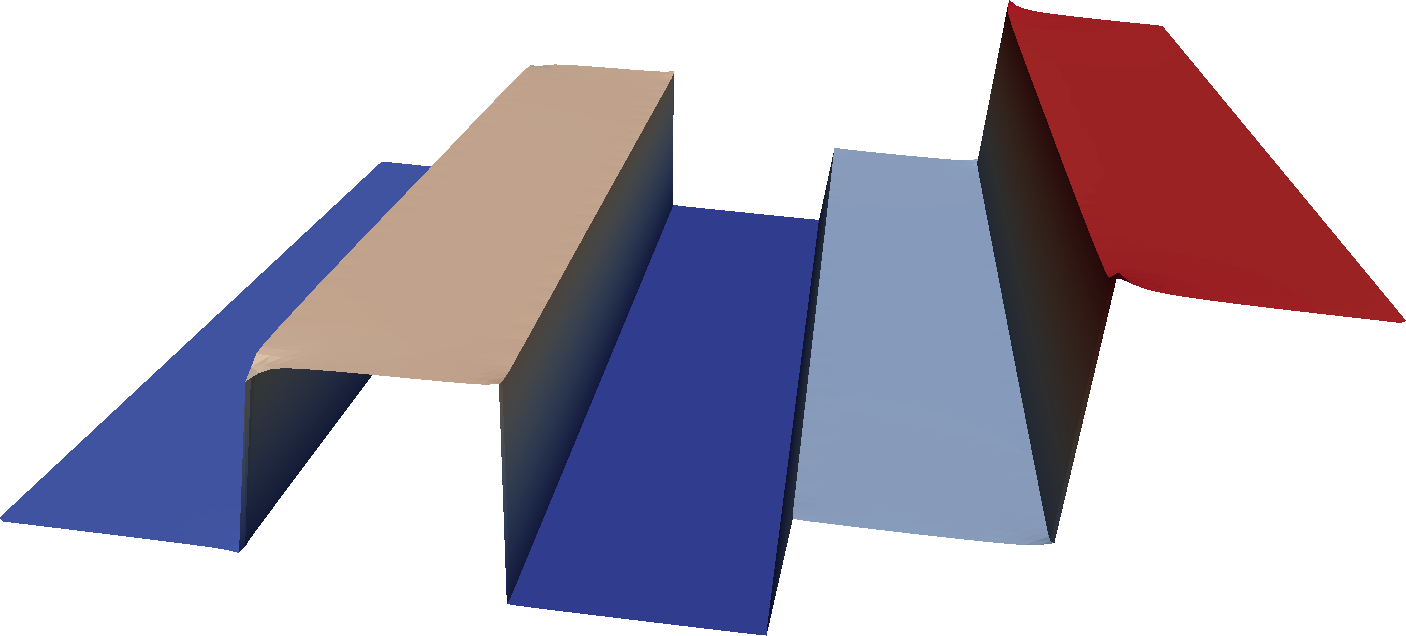}
  \includegraphics[width=0.45\textwidth]{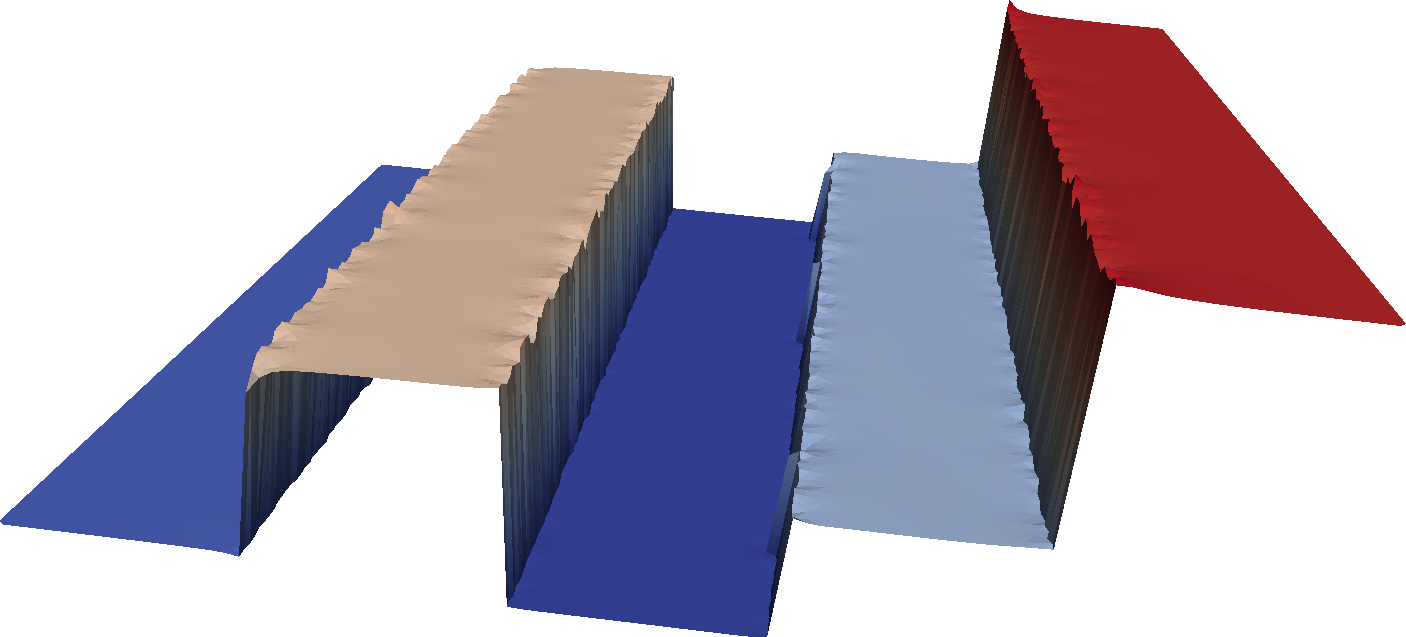} \\
    \includegraphics[width=0.45\textwidth]{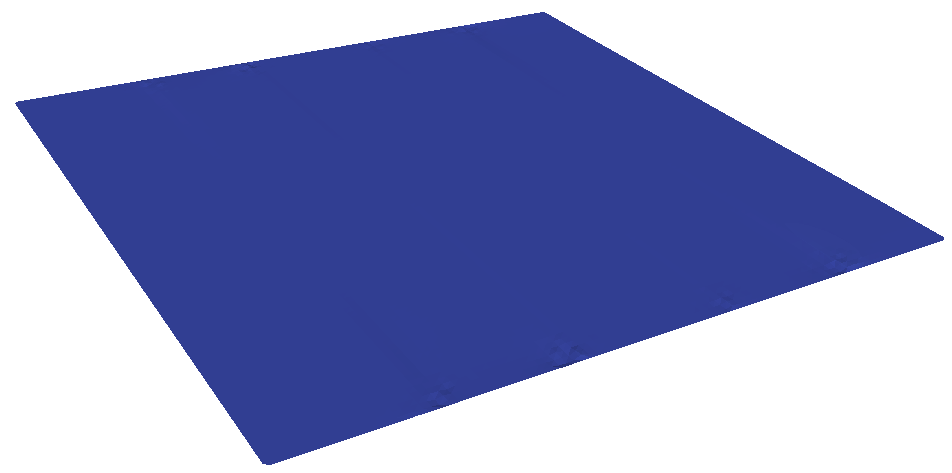}
  \includegraphics[width=0.45\textwidth]{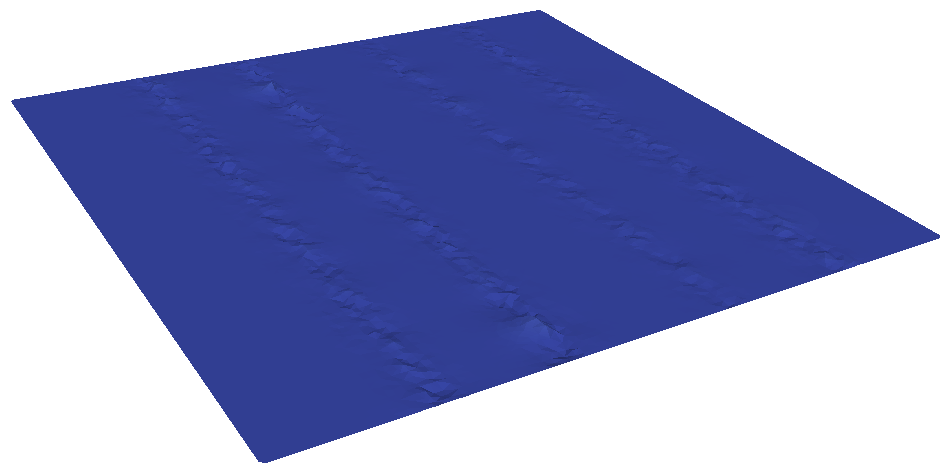}
  \caption{Approximation of the horizontal (top) and vertical (bottom) flux components by the virtual finite volume scheme \eqref{eq:ad-MMD} on a particle cloud with $h_{X} = 1/96$. \textit{Left:} uniform (Cartesian) particle distribution. \textit{Right:} non-uniform particle distribution.}
  \label{fig:5strip1}
\end{figure}

\begin{figure}[h!]
  \centering
  \includegraphics[width=0.45\textwidth]{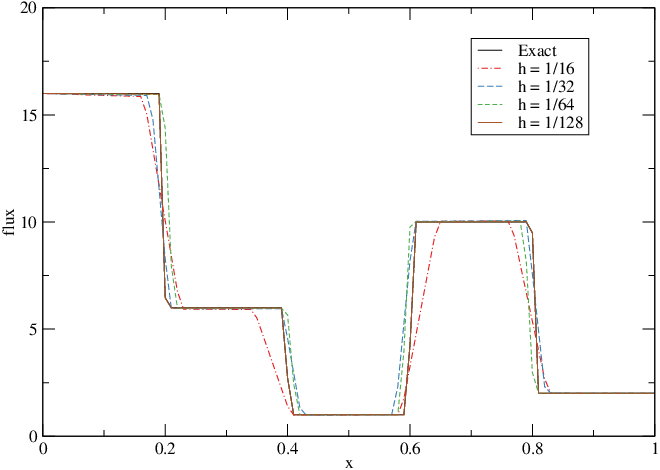}
  \includegraphics[width=0.45\textwidth]{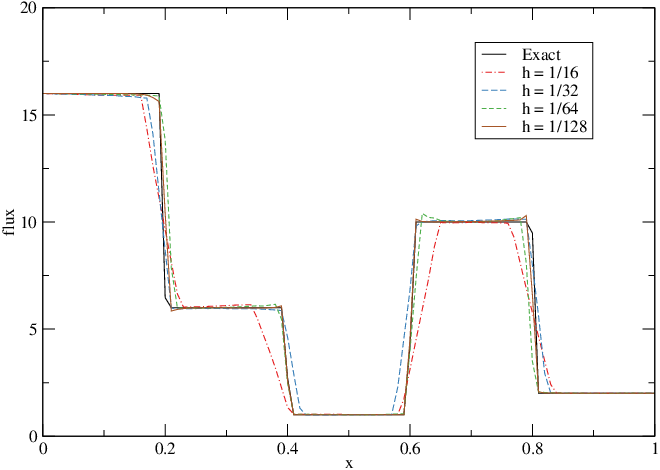}
  \caption{Profiles of the horizontal flux approximation by the virtual finite volume scheme along the line $x=0.5$ as the point cloud is refined.  
  \textit{Left:} uniform (Cartesian) particle distribution. \textit{Right:} non-uniform particle distribution.}
\label{fig:5strip2}
\end{figure}
%%%%%%%%%%%%%%%%%%%%

\paragraph{The five-spot problem}
This example models an injection-extraction application in which the flow is driven from an injection well in the bottom left corner of the domain to an extraction well in the top right. The diffusivity $\varepsilon$ is a piecewise constant which assigned the same value on the bottom left and top right subdomains and another value on the bottom right and top left subdomains; see the right plot in Fig. \ref{fig:darcyGeom}.

To drive the flow we use the same Neumann conditions as in  \cite[Figure 26]{masud}. Specifically, we apply a flux of $-\frac18$ to both faces coincident to the bottom left corner, and $\frac18$ to those coincident with the top right (Figure \ref{fig:5spot0}). In the case where $\varepsilon_1 = \varepsilon_2$, this problem may be solved analytically via superposition of Greens function solutions, assuming a periodic domain. Figure \ref{fig:5spot1} demonstrates convergence of the virtual finite volume scheme to the analytic solution for both the pressure and the flux.

\begin{figure}[h!]
  \centering
  \includegraphics[width=0.2\textwidth]{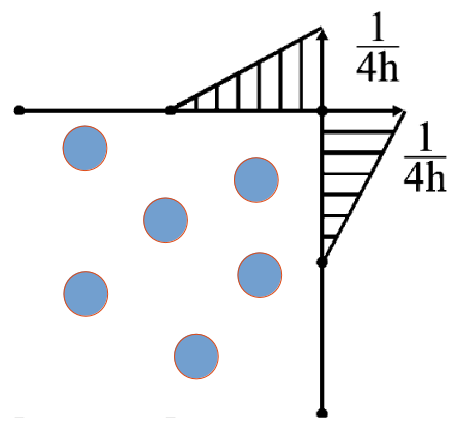}
  \caption{
  %Five spot problem with $\varepsilon_1 = \varepsilon_2$. 
  %\textit{Left:} 
Specification of a flux field yielding a total flux of $1/8$ across the faces coincident with the top right corner of the domain. A similar field is specified on the faces coincident with the bottom left corner, yielding a total flux of $-1/8$ across those faces. 
  %\textit{Right:} Streamlines of resulting flow. 
  } 
\label{fig:5spot0}
\end{figure}

\begin{figure}[h!]
  \centering
    \includegraphics[height=0.25\textwidth]{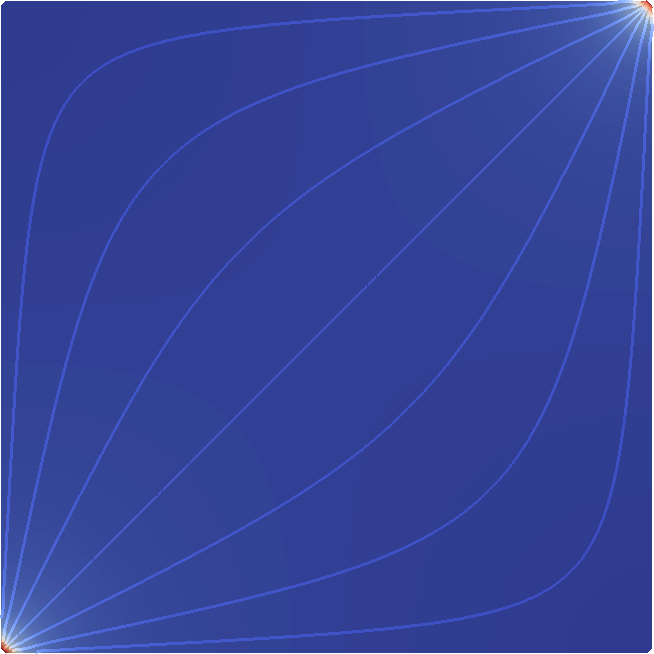}
   \includegraphics[height=0.25\textwidth]{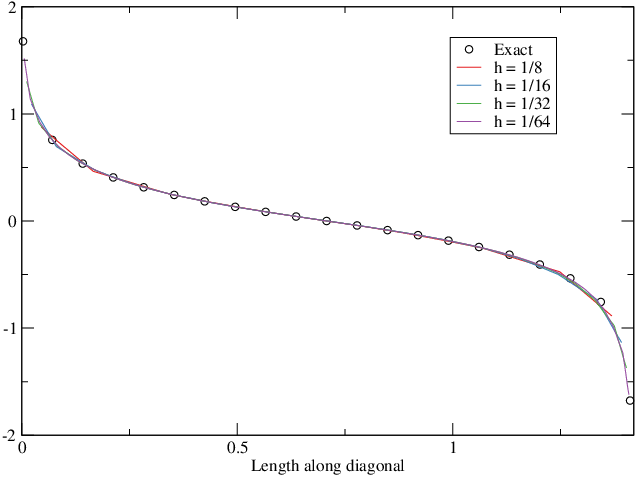}
   \includegraphics[height=0.25\textwidth]{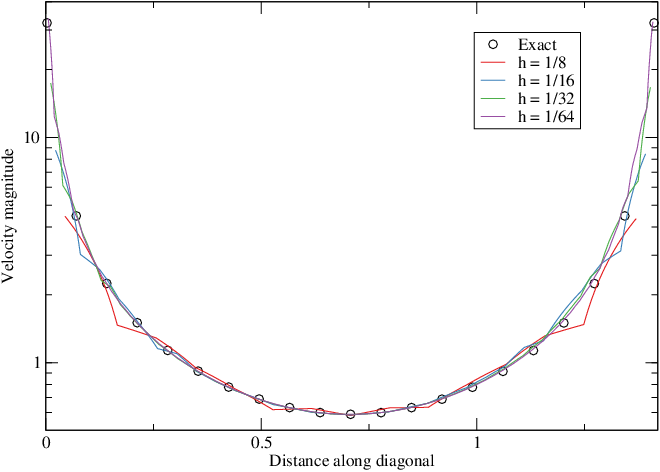}
   \caption{Comparison of the analytic solution of the five spot problem with $\varepsilon_1=\varepsilon_2=0.5$ and solutions of the virtual finite volume scheme \eqref{eq:ad-MMD} on successively refined point clouds. 
 \textit{Left:} Streamlines of flow. 
 \textit{Center:} pressure along $x=y$ diagonal. 
 \textit{Right:} flux along $x=y$ diagonal.}
 \label{fig:5spot1}
\end{figure}

\begin{figure}[h!]
  \centering
 \includegraphics[width=0.5\textwidth]{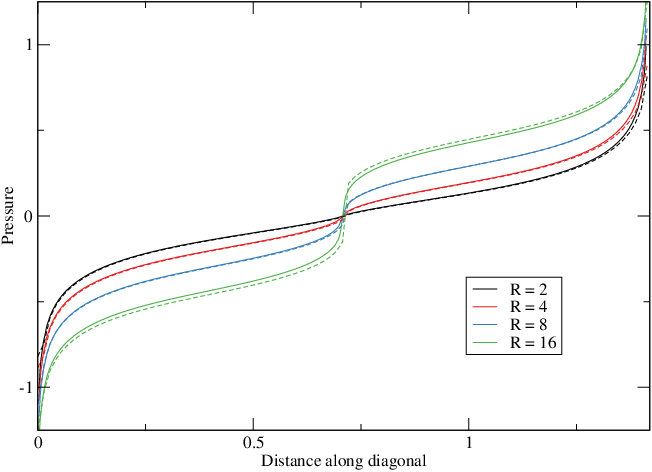}\qquad
 \includegraphics[width=0.35\textwidth]{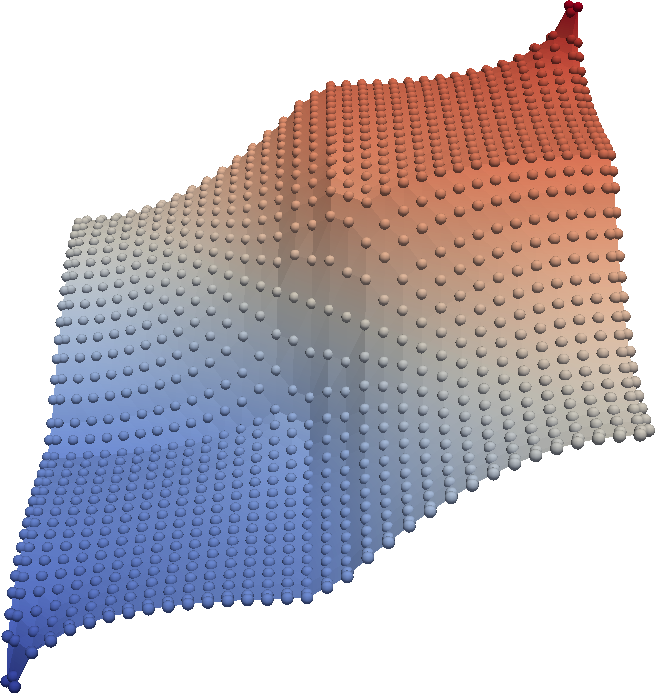}
\caption{
The five spot problem with varying $R = \varepsilon_1/\varepsilon_2$.
\textit{Left:} Comparison of the pressure profiles along $y=x$ by the virtual finite volume scheme \eqref{eq:ad-MMD} (solid lines) and a  mixed method using RT0-P0 elements (dashed lines) for increasing ratios $R$. 
\textit{Right:} Pressure approximation by the virtual finite volume scheme \eqref{eq:ad-MMD} for large ratio $R = 1000$.}
\label{fig:5spot2}
\end{figure}

To highlight the robustness of the virtual finite volume scheme we consider the five spot problem with different diffusivity ratios $R = \varepsilon_1/\varepsilon_2$. The left plot in Figure \ref{fig:5spot2} compares the pressure computed by \eqref{eq:ad-MMD} with a solution obtained by a mixed method implemented with the RT0-P0 element pair. We see that for all diffusivity ratios the two pressure profiles are nearly identical. Finally, the right plot in the same figure shows that the virtual finite volume scheme continues to deliver stable solutions even for very large material contrasts. 

The results in this section suggest that for a range of representative example problems the virtual finite volume scheme does indeed behave in a manner that mimics the behavior of mesh-based $H(div)$-conforming methods, such as mixed finite elements implemented with the RT0-P0 element pair.

%%%%%%
\subsection{Advection-diffusion problem}
%%%%%%
In this section we examine the performance of  the virtual finite volume scheme \eqref{eq:ad-MMD} for the advection-diffusion problem using a range of P\'{e}clet numbers
$$
Pe := \frac{||\mathbf{a}||}{\varepsilon} 
$$
spanning both the diffusion-dominated, i.e., $Pe = O(1)$, and the advection dominated, i.e.,  $Pe \gg O(1)$, regimes. Our objectives are twofold. First, using manufactured solutions, we aim to demonstrate the consistency and the accuracy of the virtual finite volume scheme for both the centered $\widetilde{\bm{\sigma}}_a({u}^h)$ and the upwind $\maxvec{\widetilde{\bm{\sigma}}}_{a}({u}^h)$ reconstructions of the advective flux defined in \eqref{eq:upwind}. 
Our second goal is to confirm that the construction of $\maxvec{\widetilde{\bm{\sigma}}}_{a}({u}^h)$ does provide an appropriate notion of upwinding in the advection-dominated regime.

% We do not advocate that particle methods are an ideal choice for this type of problem. Rather, we include this section to illustrate that schemes requiring notions of upwinding may be implement in this framework, despite the ambiguity of understanding precisely what is meant by upwinding at a virtual face. This scheme will be important for handling advection in a Lagrangian framework, which we will investigate in a subsequent work. 
%We first use manufactured solutions to demonstrate the consistency of the approach, using both the centered reconstruction (Equation \ref{ADcenter}) and upwind reconstruction (Equation \ref{ADupwind}) for the advection term. 

In both cases we consider the unit square with a quasi-uniform point cloud $X$ defined according to the procedure described at the beginning of \S\ref{sec:num}. To study the accuracy of the virtual finite volume scheme we use the same manufactured solution as in \S\ref{sec:convergence}, i.e., $u_{ex}(x,y)=\sin 2 \pi x \sin 2 \pi y$, and apply Dirichlet boundary conditions at all particles in $X_\Gamma$. 

\begin{table}[]\label{ADcenterconv}
\centering
\begin{tabular}{rllcc}
%\multicolumn{1}{l}{}   & \multicolumn{4}{c}{l2-norm error}                                                                  \\
${N}$      & {$Pe = 1$} & {$Pe = 10$} & {$Pe = 100$} & \multicolumn{1}{l}{{$Pe = 1000$}} \\ \hline
$16$          & 0.0535674       & 0.0959928         & n.c.                & n.c.                                     \\
$32$          & 0.0301618       & 0.060635          & n.c.                & n.c.                                     \\
$64$          & 0.0152526       & 0.0333928         & 0.0590015          & n.c.                                     \\
$128$         & 0.00756073      & 0.0175604         & 0.0314838          & n.c.                                     \\
$256$         & 0.00376267      & 0.00901208        & 0.0161104          & n.c.                                     \\ \hline
Rate      & 1.006           & 0.957             & 0.963              & n.c.                                    
\end{tabular}
\caption{Convergence rates in $\ell_2$-norm of the virtual finite volume approximation to $u_{ex}(x,y)=\sin 2 \pi x \sin 2 \pi y$ using the centered advective flux reconstruction $\widetilde{\bm{\sigma}}_a({u}^h)$.
An entry of $n.c.$ denotes that the linear solver failed to converge due to ill-conditioning.}
\label{ADcenterconv}
\end{table}

\begin{table}[]\label{ADupwindconv}
\centering
\begin{tabular}{rlllll}
${N}$             & {$Pe = 1$} & {$Pe = 10$} & {$Pe = 100$} & \multicolumn{1}{l}{{$Pe = 1000$}} & {$Pe = 10000$} \\ \hline
$16$  & 0.061851        & 0.11673           & 0.146463           & 0.153675                                & 0.15467              \\
$32$  & 0.031884        & 0.0633952         & 0.0911815          & 0.0971759                               & 0.0979677            \\
$64$  & 0.0155568       & 0.0339481         & 0.0550493          & 0.0595179                               & 0.0600997            \\
$128$ & 0.00762096      & 0.0177066         & 0.0304351          & 0.0331474                               & 0.0337972            \\
$256$ & 0.00377766      & 0.009056          & 0.0158462          & 0.0172527                               & 0.0174534            \\ \hline
Rate             & 1.02            & 0.967             & 0.944              & 0.944                                   & 0.954   
\end{tabular}
\caption{Convergence rates in $\ell_2$-norm of the virtual finite volume approximation to $u_{ex}(x,y)=\sin 2 \pi x \sin 2 \pi y$ using the upwind advective flux reconstruction $\maxvec{\widetilde{\bm{\sigma}}}_{a}({u}^h)$.}
\label{ADupwindconv}
\end{table}

Tables \ref{ADcenterconv} and \ref{ADupwindconv} present convergence rates for increasing P\'{e}clet numbers using the centered $\widetilde{\bm{\sigma}}_a({u}^h)$ and the upwind $\maxvec{\widetilde{\bm{\sigma}}}_{a}({u}^h)$ flux reconstructions, respectively.
The results in Table \ref{ADcenterconv} show that  virtual finite volume scheme with the centered advective flux reconstruction is able to maintain $O(h)$ convergence for small to moderate P\'{e}clet numbers. However, as the problem becomes strongly advection-dominated, using the centered flux leads to ill-conditioned discrete equations and failure of the iterative solver to converge. 
On the other hand, Table \ref{ADupwindconv} reveals that with the upwind reconstruction $\maxvec{\widetilde{\bm{\sigma}}}_{a}({u}^h)$ the virtual finite volume scheme remains first-order accurate over a wide range of P\'{e}clet numbers.

To further study the appropriateness of $\maxvec{\widetilde{\bm{\sigma}}}_{a}({u}^h)$ in the advection-dominated case we consider an example with discontinuous boundary data and velocity field $\mathbf{a}=(1,2)$ similar to the skew advection test in \cite{Brooks_82_CMAME}. 
To describe the problem setup let $\Gamma_t$, $\Gamma_r$, $\Gamma_b$ and $\Gamma_l$ denote the top, right, bottom, and left sides of the unit square domain, respectively.  On the bottom and left boundaries we impose Dirichlet boundary conditions given by 
$$
u  = 1\ \ \mbox{on $\Gamma_l$}
\qquad\mbox{and}\qquad
u  = 
\begin{cases}
    1,              &  x \leq \frac14 \\
    0,              &  x > \frac14
\end{cases}
\ \ \mbox{on $\Gamma_b$},
$$
respectively. On the remaining parts of the boundary we impose either the Dirichlet conditions 
$$
u = 1 \ \ \mbox{on $\Gamma_t$}
\qquad\mbox{and}\qquad
u = 0  \ \ \mbox{on $\Gamma_r$}
$$
or the following outflow condition, derived by setting $\bm{n}\cdot \bm{\sigma}_{d}(u) = 0$ in the flux definition.
\begin{equation}\label{outflow}
  \bm{n} \cdot \bm{\sigma}(u) = \bm{n} \cdot \mathbf{a} u_i
   \ \ \mbox{on $\Gamma_t \cup \Gamma_r$} .
\end{equation}

Figure \ref{fig:advdif} shows plots of the solution for both the pure Dirichlet and Dirichet/outflow cases computed by the virtual finite volume scheme with the upwind advective flux for $Pe \in \left\{1,10,100\right\}$. 
%
% Need to compare upwind with centered for 1,10 and 100. 
%
A solution plot for the Dirichet/outflow boundary conditions and $Pe = 1000$ is shown on Figure \ref{fig:advdif2}. We see that in this case the solution develops moderate crosswind oscillations that remain localized near the internal layer. This behavior is similar to the one observed in, e.g., streamline upwind finite element methods \cite{Knobloch_08_ETNA} and can be corrected by including an appropriate discontinuity capturing term \cite{Hughes_86a_CMAME}. However, design of an advective flux reconstruction incorporating such a term is beyond the scope of this paper.  
%%%%%%%%%%%%%%%%%
\begin{figure}[t]
  \centering
 \includegraphics[width=0.9\textwidth]{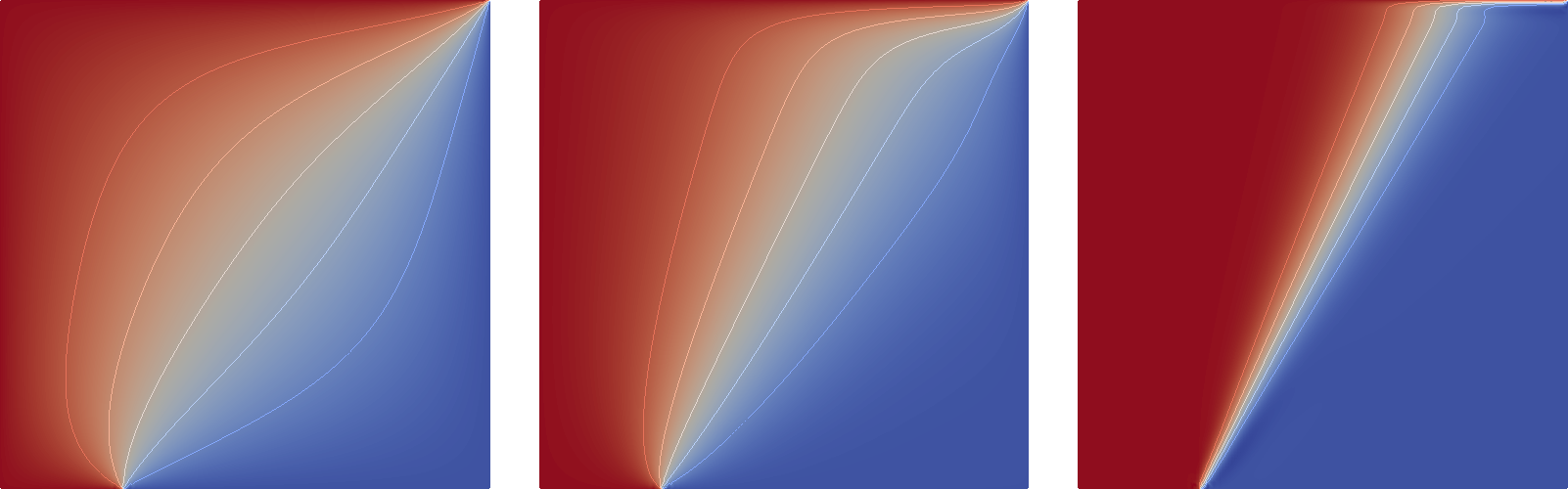}\\
 \includegraphics[width=0.9\textwidth]{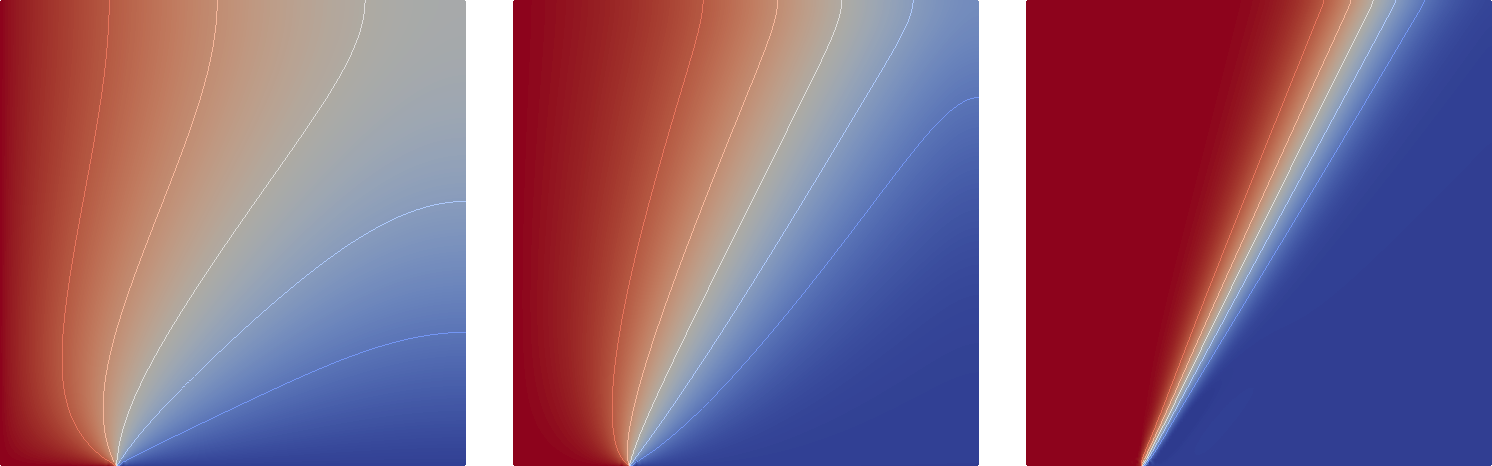}
\caption{Solution of the skew advection example by the virtual finite volume scheme with the upwind flux $\maxvec{\widetilde{\bm{\sigma}}}_{a}({u}^h)$ using  Dirichlet (\textit{top}) and Dirichlet/outflow (\textit{bottom}) boundary conditions.
 \textit{Left:} $Pe = 1$. \textit{Center:} $Pe = 10$.  \textit{Right:} $Pe = 100$.}
 \label{fig:advdif}
\end{figure}
%%%%%%%%%%%%%%%%%
\begin{figure}[t]
  \centering
 \includegraphics[width=0.35\textwidth]{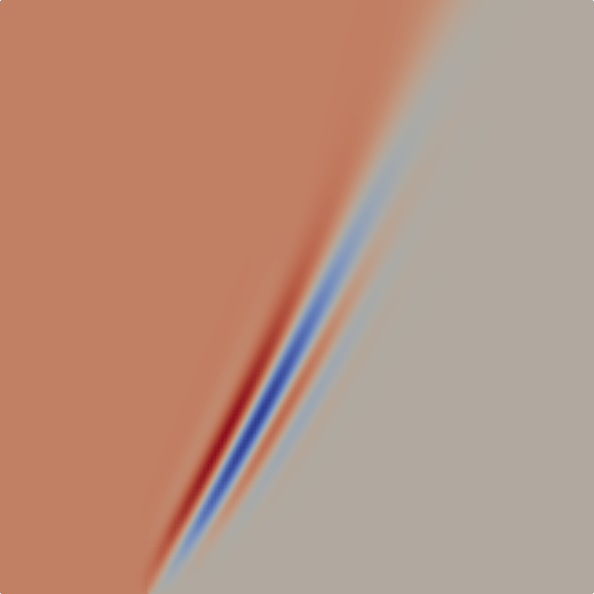}
\caption{Solution of the skew advection example by the virtual finite volume scheme with the upwind flux $\maxvec{\widetilde{\bm{\sigma}}}_{a}({u}^h)$ using  Dirichlet/outflow (\textit{bottom}) boundary conditions. $Pe=1000$.}
\label{fig:advdif2}
\end{figure}

%%%%%%
\subsection{Sensitivity of MMD with respect to virtual volume definitions}\label{sec:volume-sens}
%%%%%%
We conclude by examining the impact of the virtual volume definition on the truncation error of the MMD operator \eqref{eqn:SBPdiv}.  For this study we use the smooth vector field $\bm{u} = 2 \pi \left( \cos 2 \pi x \sin 2 \pi y, \sin 2 \pi x \cos 2 \pi y \right)$, the non-trivial geometry illustrated in Figure \ref{fig:holey}, and quasi-uniform point clouds generated according to the procedure described earlier in \S\ref{sec:num}. 
Table \ref{truncation} shows the truncation errors and the corresponding convergence rates for $\DIVGMLSP$ implemented with the uniform volumes $ \mu_i^{\mbox{\tiny\sf U}} $ and the non-uniform volumes $ \mu_i^{\mbox{\tiny\sf MR}} $, respectively. The results in this table show that in both cases the MMD operator remains first-order accurate.
%
%%%%%%%%%%%%%%%%
\begin{table}[]
\centering
\begin{tabular}{ccc}
$N$            & $ \mu_i^{\mbox{\tiny\sf U}} $ & $ \mu_i^{\mbox{\tiny\sf MR}} $
%& \textbf{Strategy 3} 
\\ \hline
$16$  & 0.312637            & 0.267592            
%& 0.204187            
\\
$32$  & 0.188097            & 0.145183            
%& 0.107237          
\\
$64$  & 0.102245            & 0.076769            
%& 0.053306            
\\
$128$ & 0.054184            & 0.040412            
%& 0.025129            
\\
$256$ & 0.027495            & 0.020119            
%& 0.013266            
\\ \hline
Rate            & 0.978711            & 1.00622             
%& 0.921556           
\end{tabular}
\caption{Truncation errors in $\ell_2$-norm and convergence rate of the MMD operator \eqref{eqn:SBPdiv} implemented using the uniform \eqref{eqn:vol1} and non-uniform \eqref{eqn:vol2}  virtual volume definitions. 
}
\label{truncation}
\end{table}
%%%%%%%%%%%%%%%%

%%%%%%%%%%%%%%%%%%
\section{Conclusion}\label{sec:concl}
%%%%%%%%%%%%%%%%%%
%
The principal contribution of this work is the formulation of a consistent and conservative mimetic meshfree divergence operator that can be used with or without a background grid while remaining computationally efficient in the latter case. 
Specifically, implementation of the operator in the absence of  background mesh involves solution of graph Laplacian problems, which can be accomplished in a scalable and efficient manner by using standard multigrid techniques. 
We demonstrate the accuracy and robustness of our operator by using it to define a virtual finite volume scheme for a scalar advection-diffusion equation. Our numerical results show that the scheme is first-order accurate, exhibits the same traits as an $H(div)$-conforming mesh-based discretization and can be equipped with an appropriate notion of upwinding for advection-dominated problems.

\section*{Acknowledgments}
This material is based upon work supported by the U.S. Department of Energy,
Office of Science, Office of Advanced Scientific Computing Research under Award
Number DE-SC-0000230927, and the Laboratory Directed Research and Development
program at Sandia National Laboratories.  The first author acknowledges support through the NSF-MSPRF program.





%\begin{thebibliography}{00}
%\bibitem{}
%\end{thebibliography}

\section*{References}
\bibliographystyle{elsarticle-num-names}
\bibliography{trask18jcp}

\end{document}